\documentclass[leqno]{article}
\usepackage{palatino} 
\usepackage{verbatim,amsmath,amsfonts,amssymb,array,theorem}
\usepackage[mathbf,mathcal]{euler}
\usepackage{pifont,fancyheadings_ssv,wrapfig,graphicx}
\usepackage{epstopdf}
\usepackage[usenames,dvipsnames]{xcolor}
\usepackage{tikz, tikz-cd}
\usepackage{extarrows}
\usepackage{enumitem}
\usepackage{mathtools}
\usepackage[top=1in, bottom=1in, outer=1in, inner=1in, heightrounded, marginparwidth=1.5cm, marginparsep=1cm]{geometry}
 \usepackage[savepos]{zref}
\usetikzlibrary{matrix, calc, positioning}
\usepackage{thmtools}
\usepackage{thm-restate}
\usepackage{pgfplots}
\usepackage{subcaption}
\usepackage{hyperref}
\usepackage{bm}

\let\frak=\mathfrak

\renewcommand{\emptyset}{\mathop{\varnothing}} 

\newcommand{\qed}{\mbox{\hfill$\square$}}

\newcounter{acount}

\newcounter{Example}
\newenvironment{example}%
    {\renewcommand{\theExample}{\arabic{Example}}\refstepcounter{Example}%
     \begin{list}{}%
        {\usecounter{Examples}%
         \setlength{\labelsep}{0pt}\setlength{\leftmargin}{0pt}%
         \setlength{\labelwidth}{0pt}\setlength{\listparindent}{0.5in}%
        }%
     \item[\normalfont\textsc{Example} \arabic{Example}) ]%
     \renewcommand{\theExample}{\arabic{Example}}%
    }
    {\end{list}}
\newcounter{Examples}
\newcommand{\firstexample}%
    {\renewcommand{\Lbl}{\textsc{Examples: }}%
     \setcounter{Examples}{\theExample}%
    }

    {\renewcommand{\theExample}{\arabic{Example}}%
     \setcounter{Examples}{\theExample}%
     \refstepcounter{Examples}%
     \begin{list}{\Lbl\arabic{Examples}) }%
        {\usecounter{Examples}%
         \setlength{\labelsep}{0pt}\setlength{\leftmargin}{0pt}%
         \setlength{\labelwidth}{0pt}\setlength{\listparindent}{0.5in}%
        }%
    }%
    {\setcounter{Example}{\theExamples}%
     \renewcommand{\theExample}{\arabic{Example}}\end{list}%
    }   
\newcommand{\Lbl}{}                            

\newenvironment{proof}{%
    \medskip\noindent\textsc{Proof}:
  }{
    \hfill\qed\medskip
  }
\newcounter{bbean}
  \newcounter{Remarks}
  \newenvironment{remark}%
    {\renewcommand{\theRemark}{\arabic{Remark}}\refstepcounter{Remark}%
     \begin{list}{}%
        {\usecounter{Remarks}%
         \setlength{\labelsep}{0pt}\setlength{\leftmargin}{0pt}%
         \setlength{\labelwidth}{0pt}\setlength{\listparindent}{0.5in}%
        }%
     \item[\normalfont\textsc{Remark} \arabic{Remark}) ]%
     \renewcommand{\theRemark}{\arabic{Remark}}%
    }
    {\end{list}}
  \newcounter{Remark}
\newcommand{\firstremark}%
    {\renewcommand{\Lbl}{\textsc{Remarks: }}%
     \setcounter{Remarks}{\theRemark}%
    }

    {\renewcommand{\theRemark}{\arabic{Remark}}%
     \setcounter{Remarks}{\theRemark}%
     \refstepcounter{Remarks}%
     \begin{list}{\Lbl\arabic{Remarks}) }%
        {\usecounter{Remarks}%
         \setlength{\labelsep}{0pt}\setlength{\leftmargin}{0pt}%
         \setlength{\labelwidth}{0pt}\setlength{\listparindent}{0.5in}%
        }%
    }%
    {\setcounter{Remark}{\theRemarks}%
     \renewcommand{\theRemark}{\arabic{Remark}}\end{list}%
    }   
\theoremstyle{plain}
{\theorembodyfont{\rmfamily}
}
{\theorembodyfont{\rmfamily} \newtheorem{definition}{Definition}}
\newtheorem{lemma}{Lemma}                      

{\theorembodyfont{\rmfamily} }
{\theorembodyfont{\rmfamily} }
{\theorembodyfont{\rmfamily} }

\theoremheaderfont{\scshape}
 \hypersetup{
    colorlinks,
    citecolor=blue,
    filecolor=blue,
    linkcolor=blue,
    urlcolor=blue
}

 \newcommand{\bigslant}[2]{{\raisebox{.2em}{$#1$}\left/\raisebox{-.2em}{$#2$}\right.}}

\newcommand{\ZZ}{\mathbb{Z}}

\newcommand{\RR}{\mathbb{R}}

\newcommand{\CC}{\mathbb{C}}
\newcommand{\NN}{\mathbb{N}}

\newcommand{\KK}{\mathbb{K}}

\newcommand{\dd}{\partial}

\newcommand{\la}{\langle}
\newcommand{\ra}{\rangle}
\newcommand{\cg}[1]{{\cal #1}}

\newcommand{\Crit}{\text{Crit}}

\tikzcdset{arrow style=tikz}

\begin{document}
\thispagestyle{plain}
\title{Rabinowitz Floer homology and mirror symmetry}
\author{Sara Venkatesh}
\date{}
\maketitle
\begin{abstract}
\noindent We define a quantitative invariant of Liouville cobordisms with monotone filling through an action-completed symplectic cohomology theory.  We illustrate the non-trivial nature of this invariant by computing it for annulus subbundles of the tautological bundle over $\CC P^1$ and give further conjectural computations based on mirror symmetry.  We prove a non-vanishing result in the presence of Lagrangian submanifolds with non-vanishing Floer homology.
\end{abstract}
\tableofcontents

\section{Introduction}
\label{introduction}
In this paper we introduce a quantitative invariant of symplectic cobordisms between contact manifolds, and we use ideas inspired by mirror symmetry to compute a non-trivial example.  When the cobordism is the trivial identity cobordism of a contact manifold, this invariant is a version of the Rabinowitz Floer homology studied by Cieliebak-Frauenfelder-Oancea in \cite{cieliebak-f-o} and akin to the Rabinowitz Floer homology of negative line bundles studied by Albers-Kang in \cite{albers-k}.  In general, computing Rabinowitz Floer homology is difficult; one of the main ideas of this paper is that mirror symmetry provides conjectural computations of our invariant.

To set the context for this invariant, recall that an embedding $V\subset M$ of Liouville domains engenders a map from the symplectic homology $SH_*(V)$ of $V$ to the symplectic cohomology $SH^*(M)$ of $M$.  In \cite{cieliebak-o}, Cieliebak-Oancea defined the symplectic cohomology $SH^*(W)$ of the Liouville cobordism $W := M\setminus V$ to measure how far from isomorphism this map strays.  That is, there is a long-exact sequence
\begin{equation}
...\longrightarrow SH_i(V)\longrightarrow SH^i(M)\longrightarrow SH^i(W)\longrightarrow SH_{i+1}(V)\longrightarrow...
\label{eq:les-co}
\end{equation}
By \cite{cieliebak-f-o}, this construction specialized to the core of a trivial cobordism recovers the Rabinowitz Floer homology.

Suppose instead that $V$ and $M$ are monotone symplectic domains with positive contact boundary, while $W$ remains a Liouville cobordism.  Taking coefficients in the universal Novikov field $\Lambda$ (\ref{eq:novikov}), one can consider the action-completed symplectic cohomology $\widehat{SH^*}(M; \Lambda)$ of $M$ and action-completed symplectic homology $\widehat{SH_*}(V; \Lambda)$ of $V$.  In Section \ref{sh(w)}, we define the {\it completed symplectic cochain complex $\widehat{SC^*}(W; \Lambda)$ of $W$}.  We denote its homology by $\widehat{SH^*}(W; \Lambda)$, so that, analogously to (\ref{eq:les-co}),

\begin{restatable}{thm}{les}
\label{thm:les}
There is a long-exact sequence
\[
...\longrightarrow \widehat{SH_i}(V; \Lambda)\longrightarrow \widehat{SH^i}(M; \Lambda)\longrightarrow \widehat{SH^i}(W; \Lambda)\longrightarrow \widehat{SH_{i+1}}(V; \Lambda)\longrightarrow...
\]
\end{restatable}

The construction of a chain complex that computes the symplectic cohomology of a cobordism is new, as is the consideration of action-completed cohomologies in this context.  The former extends the telescope construction of Abouzaid and Seidel \cite{abouzaid-s}.  The latter, although analogous to the uncompleted theory, has unexpected quantitative properties unseen in the set-up considered in \cite{cieliebak-o}.  Cobordisms in the tautological line bundle over $\CC P^1$ illustrate these properties.

\begin{restatable}{thm}{examplee}
\label{prop:example}
Let $E$ be the total space of the line bundle $\cg{O}(-1)\longrightarrow \CC P^1$ with area one exceptional divisor.  Let $W$ be a cobordism in $E$ between a sphere bundle of radius $R_1$ (possibly empty) and a sphere bundle of radius $R_2$, with $R_1 \leq R_2$.  Then
\[
\widehat{SH^*}(W; \Lambda)\cong\left\{
\begin{array}{cc}
\Lambda & R_1 \leq \frac{1}{\sqrt{\pi}} \leq R_2 \\
0 & \text{ otherwise}
\end{array}\right..
\]
\end{restatable}
Thus, $\widehat{SH^*}(W; \Lambda)$ is non-zero if and only if $W$ contains the sphere bundle of radius $\frac{1}{\sqrt{\pi}}$.  This is surprising, as uncompleted theories defined on domains of finite ``radius'', are isomorphic to the theories defined at ``infinite radius''.  Uncompleted theories therefore capture symplectic information common to domains of all radii.  In contrast, $\widehat{SH^*}(W; \Lambda)$ captures symplectic information unique to $W$: Smith showed in \cite{smith} that the  sphere bundle of radius $\frac{1}{\sqrt{\pi}}$ contains a monotone, Floer-theoretically essential Lagrangian torus $L$; moreover, Ritter-Smith showed that this Lagrangian split-generates the wrapped Fukaya category of $Tot(\cg{O}(-1)\longrightarrow \CC P^1)$ \cite{ritter-s}.  As is now suggested by Theorem \ref{prop:example}, the existence of such Lagrangians within a cobordism $W$ is intimately tied to the non-vanishing of $\widehat{SH^*}(W; \Lambda)$.

\begin{restatable}{thm}{nonvanish}
\label{thm:nonvanish}
Let $M$ be a monotone symplectic manifold and $W\subset M$ a Liouville cobordism.  Suppose that $W$ contains a compact, oriented monotone Lagrangian $L$.  If $L$ admits a flat line bundle $E_{\gamma}$ such that the Floer homology $HF^*(L, E_{\gamma})\neq 0$, then 
\[
\widehat{SH^*}(W; \Lambda)\neq 0.
\]
If $\Lambda$ is defined over a coefficient field of characteristic not equal to two, we also require the Lagrangian to be spin.
\end{restatable}

In view of Theorem \ref{thm:nonvanish}, we conjecture that Theorem \ref{prop:example} generalizes to Liouville cobordisms between sphere subbundles of line bundles $\cg{O}(-k)\longrightarrow\CC P^m$, where $1\leq k \leq m$.  
\begin{restatable}{conjecture}{conjannuli}
\label{conj:annuli}
Suppose that $W$ is an annulus subbundle in $Tot(\cg{O}(-k)\longrightarrow\CC P^m)$ between two sphere bundles of radii $R_1$ and $R_2$, with normalized symplectic form.  Then
\[
\widehat{SH^*}(W; \Lambda)\cong\left\{
\begin{array}{cc}
 \bigslant{\Lambda[z]}{\left(1 - (-k)^kTz^{-1-m+k}\right)} & R_1 \leq  \frac{1}{\sqrt{\pi(1+m-k)}} \leq R_2 \\
0 & \text{otherwise}
\end{array}\right..
\]
\end{restatable}
Indeed, by work of Ritter-Smith \cite{ritter-s}, any such cobordism containing the sphere bundle of radius $\frac{1}{\sqrt{\pi(1+m-k)}}$ satisfies the conditions of Theorem \ref{thm:nonvanish}; this generalizes the non-vanishing part of Theorem \ref{prop:example}.  A generalization of the vanishing part of Theorem \ref{prop:example} seems harder to obtain, but is strongly indicated by mirror symmetry.  We expect that the mirror of a Liouville cobordism $W$ between sphere bundles in $Tot(\cg{O}(-k)\longrightarrow\CC P^m)$ is a subset of an appropriate rigid analytic space cut out by affinoid domains, and equipped with a superpotential.  We expect that $\widehat{SH^*}(W; \Lambda)$ vanishes precisely when the critical locus of the superpotential does not intersect the mirror of $W$.  In future work we will explore the homological mirror symmetry correspondence between these objects.  In particular, using closed mirror symmetry predictions, one could hope to compute Rabinowitz Floer homology through analyzing the ring of functions and the superpotential on the mirror.

\subsection*{Outline}
This paper is organized as follows: in Section \ref{completed} we first recall Hamiltonian Floer theory and fix notation and conventions (Subsection \ref{set-up}).  We proceed to define the completed symplectic cochain complex and the completed symplectic chain complex that compute completed symplectic cohomology and homology (Subsections \ref{cohomology} and \ref{homology}).  We then warm up by defining completed Rabinowitz Floer homology (Subsection \ref{rfh}).  Finally, we define the completed symplectic cohomology of a Liouville cobordism and show that Theorem \ref{thm:les} follows as a consequence of construction (Subsection \ref{sh(w)}).  In Section \ref{non-vanish} we prove Theorem \ref{thm:nonvanish} and in Section \ref{example} we prove Theorem \ref{prop:example}.  In Section \ref{hms} we discuss closed mirror symmetry and Conjecture \ref{conj:annuli}.

\subsection*{Acknowledgements}

I thank my advisor, Mohammed Abouzaid, for his exceptional guidance.  I thank Peter Albers, Yoel Groman, Jungsoo Kang, Jo Nelson, and Paul Seidel for helpful comments and discussions.  This work was supported by NSF grant DGE-16-44869 and Simons Foundation grant ``Homological Mirror Symmetry and Applications''.

\section{Completing Floer cochains}
\label{completed}

\subsection{Hamiltonian Floer theory on monotone manifolds}
\label{set-up}
Let $M$ be a compact symplectic manifold of dimension $2n$, equipped with symplectic form $\omega$.  Under favourable conditions one can define the Floer theory of $M$ as a homology theory on the loop space $\cg{L}M$ of $M$.  In this paper we assume three conditions that, in conjunction, prove exceptionally favorable.  The first condition requires $M$ to be monotone: there exists a constant $c > 0$ satisfying
\[
c_1^{TM} = c[\omega].
\]
The second condition requires the boundary of $M$ to be contact.  Thus, there is a one-form $\lambda$ defined near the boundary of $M$ satisfying $d\lambda = \omega$, and such that $\lambda\big|_{\dd M}$ is a contact form on $\dd M$.  
The final condition requires the boundary orientation of $\dd M$ to match the contact orientation induced by $\lambda\big|_{\dd M}$.  This is equivalent to asking that the Liouville flow $X_{\lambda}$ defined by $\omega(X_{\lambda}, -) = \lambda$ points outwards along the boundary.

Given a suitable Hamiltonian function $H: M\times S^1\longrightarrow \RR$ (defined in Section \ref{cohomology}) , one can define a Floer cohomology theory as follows.  Define the set of closed orbits of $H$ to be 
\[
\cg{P}(H) = \left\{x\in\cg{C}^{\infty}(S^1, M)\hspace{.1cm}\big|\hspace{.1cm}\dot{x} = X_H(x)\right\}.
\]
Choose a basepoint $\beta$ for each connected component of $\cg{L}M$.  Then $\cg{P}(H)$ decomposes as a direct sum
\[
\cg{P}(H) = \bigoplus_{[\beta]} \cg{P}_{\beta}(H), 
\]
where $\cg{P}_{\beta}(H) = \left\{x\in\cg{P}(H)\big| [x] = [\beta]\right\}$.
For each $x\in \cg{P}_{\beta}(H)$ choose a path $\tilde{x}$ from $x$ to $\beta$.

Fix a coefficient ring $\KK$.
Define a ring over a formal variable $T$ by
\[
\Gamma = \left\{\sum_{j=0}^n a_jT^{k_j}\bigg| a_j\in\KK, k_j\in\RR, n\in\NN\right\}.
\]

We will define the structure of a cochain complex on the set $\Gamma\left\la\cg{P}(H)\right\ra$.  Let $o_x$ be the orientation line associated to $x$ (see Subsection 1.4 in \cite{abouzaid} for a detailed account of orientation lines).  Define
\begin{equation}
\label{complex-orientation}
CF^*(H; \Gamma) = \bigoplus_{x\in\cg{P}(H)}\Gamma\otimes_{\ZZ} o_x.
\end{equation}
The Conley-Zehnder index $\mu_{CZ}$ gives $\cg{P}(H)$ a well-defined $\ZZ/2\ZZ$ - grading, and we grade $\Gamma$ trivially by setting $|T| = 0$.  Elements $\zeta_x\in o_x$, for $x\in\cg{P}(H)$, are then graded by $|\zeta_x| = \mu_{CZ}(x)\in\ZZ/2\ZZ$.  If $\KK = \ZZ/2\ZZ$ one can replace each orientation line $o_x$ in (\ref{complex-orientation}) by the corresponding periodic orbit $x$.  See \cite{salamon} for details on the Conley-Zehnder index.

The negative flow of $X_{\lambda}$ defines a collar neighborhood $[-\epsilon, 0]\times\dd M$ of the boundary of $M$, on which $\omega_{(r, x)} = e^rdr\wedge\lambda_x + e^rd\lambda_x$ (where $r$ is the coordinate on $[-\epsilon, 0]$).  Let $J$ be an almost-complex structure on $M$ that is {\it cylindrical} on the collar neighborhood of $\dd M$.  Recall that a cylindrical almost-complex structure satisfies
\[
e^rdr = J^*\lambda.
\]
We will always choose our cylindrical almost-complex structures to be $\omega$-compatible.  

Let $w:\RR\times S^1\longrightarrow M$ satisfy Floer's equation
\begin{equation}
\label{eq:floer}
\frac{\dd w}{\dd s} + J\left(\frac{\dd w}{\dd t} - X_H\right) = 0,
\end{equation}
where the cylinder $\RR\times S^1$ has coordinates $(s, t)$.  Associated to such maps is the {\it energy}, defined by
\begin{equation}
\label{eq: energy}
E(w) = \frac{1}{2}\int_{\RR\times S^1} ||dw - X_H\otimes dt||^2ds\wedge dt.
\end{equation}
If the energy of $w$ is finite, $w(s, \cdot)$ converges asymptotically in $s$ to periodic orbits of $X_H$.  For any two periodic orbits $x_-$ and $x_+$, define $\cg{M}^0(x_-, x_+)$ to be the space of rigid solutions $w(s, t)$ of (\ref{eq:floer}) satisfying $\lim\limits_{s\rightarrow \pm\infty}u(s, \cdot) = x_{\pm}(\cdot)$.  Each such $w$ has an $\RR$-action of translation in the $s$-direction, and modding out by this $\RR$-action produces a compact zero-dimensional moduli space $\hat{\cg{M}}^0(x_-, x_+)$.  Each $w$ further induces an isomorphism $d_w:o_{x_+}\longrightarrow o_{x_-}$ of orientation lines (see Lemma 1.5.4 of \cite{abouzaid}).  We denote by $-\widetilde{x_+}\#w\#\widetilde{x_-}$ the element of $\pi_2(M)$ formed by gluing $w$ to $\widetilde{x_-}$ and $\widetilde{x_+}$, the latter with reversed orientation.  For $A\in\pi_2(M)$, we will use the shorthand
\[
\omega(A) := \int_{S^2} A^*\omega.
\]
Equip $CF^*(H; \Gamma)$ with the differential $\dd^{fl}$ given on generators by
\[
\dd^{fl}\big|_{o_{x_+}} = \sum_{\substack{x_-\in\cg{P}(H)\\ w\in\hat{\cg{M}}^0(x_-, x_+)}} T^{\omega(-\widetilde{x_+}\#w\#\widetilde{x_-})}\cdot d_w.
\]
As $M$ is monotone, $\dd^{fl}$ is well-defined.  Extending the differential $T$-linearly yields the cochain complex $CF^*(H; \Gamma)$. 

Define an action functional $\cg{A}_H$ on $\left\{ T^{\alpha}x\right\}_{\alpha\in\RR, x\in\cg{P}(H)}$ by
\begin{align*}
\cg{A}_H(T^{\alpha}x) = \alpha - \int_{S^1\times[0,1]}\tilde{x}^*\omega + \int_{S^1}H(x(t)) dt
\end{align*}
and set $\cg{A}_H(T^{\alpha}\zeta_x) = \cg{A}_H(T^{\alpha}x)$ for any $\zeta_x\in o_x$.
A standard computation shows that the differential increases $\cg{A}_H$.  Thus, the subsets
\[
CF^*_a(H; \Gamma) := \KK\left\la\left\{T^{\alpha}\zeta_x\hspace{.1cm}\big|\hspace{.1cm} \alpha\in\RR; \hspace{.05cm} \zeta_x\in o_x; \hspace{.05cm} x\in\cg{P}(H); \hspace{.05cm} \cg{A}_H(T^{\alpha}\zeta_x) > a\right\}\right\ra
\]
are subcomplexes and form a filtration of $CF^*(H;\Gamma)$.

For $a < b$ define the quotient complex
\[
CF^*_{(a, b)}(H; \Gamma) := \bigslant{CF^*_a(H; \Gamma)}{CF^*_b(H; \Gamma)}.
\]
There are natural chain maps 
\begin{equation}
CF^*_{(a, b)}(H; \Gamma)\hookrightarrow CF^*_{(a', b)}(H;\Gamma) \text{ and } CF^*_{(a, b)}(H; \Gamma)\twoheadrightarrow CF^*_{(a, b')}(H;\Gamma)
\label{eq:filtration}
\end{equation}
whenever $a' \leq a$ or $b' \leq b$, given by, respectively, inclusion and projection.  Following the example of \cite{cieliebak-f}, we will use this quotient complex and the natural maps of (\ref{eq:filtration}) to define a Novikov-type completion of different Floer homology theories on open manifolds.

\begin{remark}
The $(a, b)$-filtered complex is independent of lifts $x\mapsto\tilde{x}$, as choosing a different lift corresponds to rescaling $x$ by some power of $T$.
\end{remark}

\subsection{Symplectic cohomology}
\label{cohomology}
Hamiltonian Floer theory is not invariant under choice of Hamiltonian when working on manifolds with boundary.  To rectify this, one usually take a colimit over the Floer homologies of all suitable Hamiltonians.  The resulting homology theory captures information about the singular cohomology of $M$ and the positively-traversed Reeb orbits of various contact hypersurfaces in the conical completion of $M$.

We will define the colimit over a smaller class of Hamiltonians than is usual in the literature; in particular, we will require that the Reeb orbits captured in our cohomology theory cluster near $\dd M$.  This will define a Floer cohomology of $M$ (as opposed to its conical completion) that we will show displays, under completion-by-action, surprising behavior.

Choose $\epsilon_M > 0$.  The Liouville flow near the boundary of $M$ enables us to smoothly attach $[0, \epsilon_M)\times \dd M$ to $M$ via $\dd M$.  Define the enlarged manifold 
\[
\widetilde{M} = M\cup_{\dd M}[0, \epsilon_M)\times \dd M.
\]  
Choose any $\cg{C}^2$-small function $\cg{H}: M\times S^1\longrightarrow\RR$ with non-degenerate, constant time-one orbits.  Choose a sequence $\{\epsilon_n\}_{n\in\NN}$ that is monotone decreasing, bounded above by $\epsilon_M$, and converges to $0$.  Choose a family of Hamiltonians $Ad(M) = \{H^{\tau_i}:\widetilde{M}\times S^1\longrightarrow\RR\}_{i\in\ZZ}$, that we call {\it admissible}, such that 
\begin{enumerate}
\item $H^{\tau_i}\big|_{M\times S^1} = \cg{H}$ for all $i$,
\item $H^{\tau_i}\geq H^{\tau_j}$ whenever $i \geq j$,
\item $H^{\tau_i} = h^{\tau_i}(e^r)$ on $[0, \epsilon_M)\times \dd M$ for some function $h^{\tau_i}:\RR_+\longrightarrow\RR$,
\item $h^{\tau_i}$ is linear of slope $\tau_i$ on $(\epsilon_{|i|}, \epsilon_M)$\label{cond:noescape}, 
\item $\tau_i > 0$ if and only if $i \geq 0$, 
\item $|H^{\tau_i}|$ is universally bounded on one-periodic orbits, and
\item the one-periodic orbits of $H^{\tau_i}$ are transversely non-degenerate.
\end{enumerate}
Finally, require that $\tau_0$ be smaller than the smallest period of a positive Reeb orbit on $\dd M$.  See Figure \ref{fig:rfh} for a cartoon of the elements of $Ad(M)$.

\begin{figure}[htbp!]
\begin{center}
\includegraphics[scale=.4]{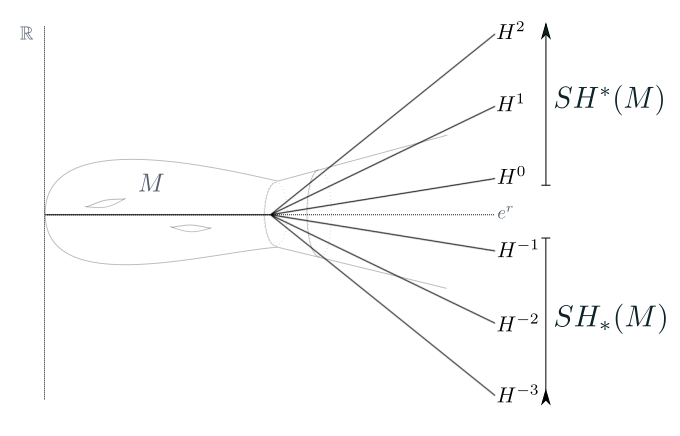}
\end{center}
\caption{The family of Hamiltonians defining completed symplectic homology and cohomology, and completed Rabinowitz Floer cohomology}
\label{fig:rfh}
\end{figure}

Define $Ad_+(M)$ to be the non-negatively-indexed Hamiltonians and $Ad_-(M)$ to be the negatively-indexed Hamiltonians.

\begin{remark}
Instead of attaching $[0, \epsilon_M)\times\dd M$ to $M$, we could have attached the entire positive symplectization $[0, \infty)\times \dd M$, and extended each $H^{\tau_i}$ linearly to define elements of $Ad(M)$ on this completed manifold.  The Floer theory of $H^{\tau_i}$ is well-defined in this setting.  Condition (\ref{cond:noescape}) and the maximum principle ensure that Floer trajectories of $H^{\tau_i}$ of finite energy, in particular the trajectories used to define the differential, do not exit $\widetilde{M}$.  All of the data used to define $CF^*(H^{\tau_i}; \Gamma)$ therefore lives in $\widetilde{M}$, and so we can ``do Floer theory'' on $\widetilde{M}$ instead of on the completed manifold.  In this paper we will only define Floer theory on manifolds of the form $\widetilde{M}$, and never on the full completed manifold.
\end{remark}

There are continuation maps $c^i:CF^*(H^{\tau_i}; \Gamma)\longrightarrow CF^*(H^{\tau_{i+1}}; \Gamma)$ for each $i$.  Again by Condition (\ref{cond:noescape}) and the maximum principle each $c^i$ is well-defined.  These maps may be chosen to respect the action filtration, thereby inducing continuation maps 
\[
c^i:CF^*_{(a, b)}(H^{\tau_i}; \Gamma)\longrightarrow CF^*_{(a, b)}(H^{\tau_{i+1}}; \Gamma).
\]
This leads to a directed system
\[
...\xrightarrow{c^{-2}} CF^*_{(a, b)}(H^{\tau_{-1}}; \Gamma) \xrightarrow{c^{-1}} CF^*_{(a, b)}(H^{\tau_{0}}; \Gamma) \xrightarrow{c^{0}} CF^*_{(a, b)}(H^{\tau_{1}}; \Gamma) \xrightarrow{c^1}...
\]
The non-negatively-indexed continuation maps induce a chain map
\[
\{c^i - id\}:\bigoplus_{i=0}^{\infty} CF^*_{(a, b)}(H^{\tau_i};\Gamma)\longrightarrow\bigoplus_{i=0}^{\infty} CF^*_{(a, b)}(H^{\tau_i};\Gamma)
\]
defined componentwise.  The cone of this map is a cochain complex 
\[SC^*_{(a, b)}(M; \Gamma) := \bigoplus_{i=0}^{\infty}CF^*_{(a, b)}(H^{\tau_i};\Gamma)\oplus\bigoplus_{i=0}^{\infty}CF^*_{(a, b)}(H^{\tau_i};\Gamma)[1]
\]
with differential given by
\[
\delta^* = \left\{\left(\begin{array}{cc} \dd^{fl} & c^i - id  \\  0& \dd^{fl}[1]\end{array}\right)\right\}.
\]

For ease of notation let ${\bm \theta}$ be a formal variable of degree $|{\bm \theta}| = -1$ satisfying ${\bm \theta}^2 = 0.$  Rewrite the symplectic chain complex as
\[
SC^*_{(a, b)}(M;\Gamma) = \bigoplus_{i=0}^{\infty} CF^*_{(a, b)}(H^{\tau_i};\Gamma)[{\bm \theta}]
\]
with differential 
\begin{equation}
\delta^*\big|_{o_x + o_y{\bm \theta}} = \dd^{fl}\big|_{o_x} + c^i\big|_{o_y} - id_{o_y} + (\dd^{fl}\big|_{o_y}){\bm \theta}.
\label{eq:diff}
\end{equation}
The maps between filtered chain complexes in equation (\ref{eq:filtration}) extend componentwise to chain maps
\[
SC^*_{(a, b)}(M; \Gamma)\hookrightarrow SC^*_{(a', b)}(M;\Gamma) \text{ and } SC^*_{(a, b)}(M; \Gamma)\twoheadrightarrow SC^*_{(a, b')}(M;\Gamma)
\]
that defines a bi-directed system.  Define the {\it completed symplectic cochains} to be the limit over this bi-directed system, and denote it by
\[
\widehat{SC^*}(M; \Gamma) = \lim_{\substack{\longrightarrow \\ a}}\lim_{\substack{\longleftarrow \\ b}}SC^*_{(a, b)}(M;\Gamma).
\]
Note that, under our conventions, the limits take $a$ to negative infinity and $b$ to positive infinity.

\begin{remark}
By Theorem 5.6 in \cite{frei-m}, taking the limits in the opposite order creates an isomorphic complex.  (Also see \cite{cieliebak-f} for an application of this theorem to Morse Theory.)
\label{rem:order}
\end{remark}

Completed symplectic cohomology is the homology of this complex, and is denoted by $\widehat{SH^*}(M; \Gamma)$.

\begin{remark}
\label{rmk:novikov}
We can write the elements of $\widehat{SC^*}(M; \Gamma)$ directly as $a + b{\bm \theta}$, where $a$ and $b$ are sums of the form
\[
\left\{\sum_{j=0}^{\infty}a_jT^{k_j}\cdot \zeta_j\hspace{.05cm}\bigg|\hspace{.05cm} a_j\in\KK; \hspace{.05cm}k_j \in \RR;\hspace{.05cm}\zeta_j\in o_{x_j}\text{ for some } x_j\in\bigcup_{H\in Ad(M)}\cg{P}(H);\hspace{.05cm} \lim_{j\rightarrow\infty}-\omega(\tilde{x_j}) + k_j = \infty\right\}.
\]
Thus, the completed symplectic cochain complex agrees with the complex formed by taking a Novikov-type completion.  In particular, it is a module over the {\it universal Novikov ring over $\KK$}, defined by
\begin{equation}
\label{eq:novikov}
\Lambda := \left\{\sum_{j=1}^{\infty} a_jT^{k_j}\bigg| a_j\in\KK; k_j\in \RR; \lim_{j\rightarrow\infty} k_j = \infty\right\}.
\end{equation}
Since the differential respects the $\Lambda$ action, we will henceforth take coefficients of completed complexes in $\Lambda$, working with $\widehat{SC^*}(M; \Lambda) := \widehat{SC^*}(M; \Gamma)$.
\end{remark}

\subsection{Symplectic homology}
\label{homology}

Symplectic homology is defined analogously to symplectic cohomology.  The negative continuation maps induce a chain map
\[
\{c^i - id\}:\prod_{i=-1}^{-\infty}CF^*_{(a, b)}(H^{\tau_i};\Gamma)\longrightarrow \prod_{i=-1}^{-\infty}CF^*_{(a, b)}(H^{\tau_i};\Gamma).
\]
Define the $(a, b)$-truncated symplectic chains to be
\[
SC_*^{(a, b)} := \prod_{i=-1}^{-\infty} CF^*_{(a, b)}(H^{\tau_i};\Gamma)[{\bm \theta}],
\]
with differential $\delta^*$ as in (\ref{eq:diff}).

Akin to Section \ref{cohomology}, the {\it completed symplectic chains} are defined to be
\[
\widehat{SC_*}(M; \Lambda) := \lim_{\substack{\longrightarrow \\ a}}\lim_{\substack{\longleftarrow \\ b}}SC_*^{(a, b)}(M; \Gamma).
\]
Completed symplectic homology is the homology of this complex, and is denoted by $\widehat{SH_*}(M; \Lambda)$.

\begin{remark}
By Poincar\'e duality, there is a chain isomorphism
\[
SC_*^{(a, b)}(M; \Gamma) \cong \prod_{i=-1}^{-\infty}CF_{-*}^{(-b, -a)}(-H^{\tau_i};\Gamma)[{\bm \zeta}],
\]
where $|{\bm \zeta}| = 1$ and ${\bm \zeta}^2 = 0$.

This implies the isomorphism
\[
\widehat{SC_*}(M; \Lambda)\cong\left(\lim_{\substack{\longrightarrow \\ a}}\lim_{\substack{\longleftarrow \\ b}}\prod_{i=-1}^{-\infty}CF_{-*}^{(-b, -a)}(-H^{\tau_i};\Gamma)[{\bm \zeta}]\right).
\]
In particular, $\widehat{SC_*}(M; \Lambda)$ is chain-isomorphic to the dual complex of the completed symplectic cochain complex (after a shift in grading), and is thereby deserving of its name, despite the cohomological conventions used to define it.
\end{remark}

\subsection{Rabinowitz Floer cohomology}
\label{rfh}

There is a map from $(a, b)$-truncated symplectic homology to $(a, b)$-truncated symplectic cohomology, given on chains by projecting onto $CF^*_{(a, b)}(H^{\tau_{-1}}; \Gamma){\bm \theta}$, applying the continuation map $c^{-1}$, and then including.  Call this map $\mathfrak{c}$.

\[
\begin{tikzcd}
SC_*^{(a, b)}(M; \Gamma) \arrow{r}{\mathfrak{c}} \arrow[twoheadrightarrow]{d}{\pi} & SC^*_{(a, b)}(M; \Gamma) \\ CF^*_{(a, b)}(H^{\tau_{-1}}; \Gamma){\bm \theta} \arrow{r}{c^{-1}} & CF^*_{(a, b)}(H^{\tau_0}; \Gamma) \arrow[hookrightarrow]{u}{\iota}
\end{tikzcd}
\]
\begin{remark}
One does not need to truncate by action; on monotone and exact domains the map $\mathfrak{c}$ extends to a map on the full complexes $SC_*(M)\longrightarrow SC^*(M)$.  It was shown in \cite{cieliebak-f-o} that the Rabinowitz Floer homology of the contact boundary of a Liouville domain is the cone of the induced map $\mathfrak{c}^*: SH_*(M)\longrightarrow SH^*(M)$.  This motivates the following definition.
\end{remark}
\begin{definition}
Define the $(a, b)$-truncated Rabinowitz Floer cochain complex $RFC^*_{(a, b)}(M)$ to be the cone of $\mathfrak{c}$:
\[
RFC^*_{(a, b)}(M):= \left(SC^*_{(a, b)}(M)\oplus SC_*^{(a, b)}(M)[1], \left(\begin{array}{cc} \delta^* & \mathfrak{c} \\ 0 & \delta^*[1]\end{array}\right)\right).
\]
\end{definition}
There is a triangle
\begin{equation}
\begin{tikzcd}
SC_*^{(a, b)}(M; \Gamma) \arrow{rr}{\mathfrak{c}} && SC^*_{(a, b)}(M; \Gamma) \arrow{dl} \\ & RFC^*_{(a, b)}(M; \Gamma) \arrow{ul}{[-1]}
\end{tikzcd}
\label{diag:exact-trunc}
\end{equation}
The inverse limit $\lim\limits_{\substack{\longleftarrow \\ b}}$ is exact (the Mittag-Leffler condition is easily satisfied via surjection of the projection maps defining the limit).  Clearly the limit $\lim\limits_{\substack{\longrightarrow \\ a}}$ is exact.  Applying the action-window limits to the triangle (\ref{diag:exact-trunc}) creates a triangle of completed complexes.
\begin{equation}
\begin{tikzcd}
\widehat{SC_*}(M; \Lambda) \arrow{rr}{\mathfrak{c}} && \widehat{SC^*}(M; \Lambda) \arrow{dl} \\ & \lim\limits_{\substack{\longrightarrow \\ a}}\lim\limits_{\substack{\longleftarrow \\b}}RFC^*_{(a, b)}(M; \Gamma) \arrow{ul}{[-1]}
\end{tikzcd}
\label{diag:exact-com}
\end{equation}

\begin{definition}
The {\it completed Rabinowitz Floer cochain complex} is 
\[
\widehat{RFC^*}(M; \Lambda) := \lim\limits_{\substack{\longrightarrow \\ a}}\lim\limits_{\substack{\longleftarrow \\b}}RFC^*_{(a, b)}(M; \Gamma)
\]  
Its homology is denoted by $\widehat{RFH^*}(M; \Lambda)$.
\end{definition}  

Note that applying homology to (\ref{diag:exact-com}) yields the exact sequence
\begin{equation}
...\longrightarrow \widehat{SH_i}(M; \Lambda)\longrightarrow \widehat{SH^i}(M; \Lambda)\longrightarrow \widehat{RFH^i}(M; \Lambda)\longrightarrow \widehat{SH_{i+1}}(M; \Lambda)\longrightarrow...
\label{diag:les-rfh}
\end{equation}

\begin{remark}
While we abuse language in calling our construction ``completed Rabinowitz Floer homology'', we expect that $\widehat{RFH^*}(M; \Lambda)$ is, after a degree adjustment, isomorphic to the Rabinowitz Floer homology found in the literature (defined for sphere bundles in negative line bundles).  \cite{albers-k}, \cite{cieliebak-f-o}, \cite{frauenfelder}.
\end{remark}

\subsection{Symplectic cohomology of a Liouville cobordism}
\label{sh(w)}

A {\it Liouville cobordism} is an exact symplectic manifold $(W, \omega = d\lambda)$ with contact boundary $(\dd W, \alpha = \lambda\big|_{\dd W})$.  If the boundary orientation of a component $B\subset \dd W$ agrees with the orientation induced by $\alpha\big|_B$, we call $B$ a {\it positive boundary component}.  If the two orientations disagree, we say that $B$ is a {\it negative boundary component}.  In general, $\dd W$ decomposes as the union of the positive boundary components $(\dd_+W, \alpha_+ = \alpha\big|_{\dd_+W})$ and negative boundary components$(\dd_-W, \alpha_- = \alpha\big|_{\dd_-W})$.

Suppose that $M$ decomposes as the union of a Liouville cobordism $W$ and a compact, monotone symplectic manifold $V$, glued along the boundary of $V$ and the negative boundary of $W$.  We will show that the map
\[
\widehat{SH_*}(M; \Lambda)\xlongrightarrow{\mathfrak{c}^*}\widehat{SH^*}(M; \Lambda)
\] 
generalizes to a map 
\[
\widehat{SH_*}(V; \Lambda)\longrightarrow\widehat{SH^*}(M; \Lambda),
\] 
and we will define the completed symplectic cohomology of $W$ analogously to completed Rabinowitz Floer cohomology.  We first fix notation and technical conventions.

As in the previous sections, we will work over $\widetilde{M} = M\cup[0, \epsilon_M)\times\dd M$.  If the flow $\Phi_{X_{\lambda}}^t(x)$ of $X_{\lambda}$ is defined for all $t\in(T_1, T_2)$ and $x\in\dd_{\pm}W$, we identify the subdomain
\[
\left\{\Phi_{X_{\lambda}}^t(x)\big| t\in(T_1, T_2), x\in \dd_{\pm}W\right\}
\]
with the subspace $(T_1, T_2)\times\dd_{\pm}W$ of the symplectization of $\dd_{\pm}W$.  Let $r$ be the coordinate on $(T_1, T_2)$ and $x$ the coordinate on $\dd_{\pm}W$.  Under this identification, $\lambda_{r, x} = e^r(\alpha_{\pm})_x$.  Fix $R > 0$ such that $\Phi_{X_{\lambda}}$ is defined on $(-R, R)\times\dd_-W$ and $(-R, \epsilon_M)\times\dd_+W$, and
\[
\left\{(-R, R)\times\dd_-W\right\} \cap \left\{(-R, \epsilon_M)\times\dd_+W\right\} = \emptyset.
\]
Let $W_+ = \left(W\cup[0, \epsilon_M)\times\dd_+W\right)\setminus[0, R)\times\dd_-W$ 


In the previous section we considered the set of admissible Hamiltonians $Ad(M) = Ad_+(M)\sqcup Ad_-(M)$.  Leave the subfamily $Ad_+(M)$ unchanged and redefine $Ad_-(M)$ as follows.  Choose $\epsilon_V\in(0, R)$.  Let $\{\epsilon_i\}_{i\in\ZZ_{< 0}}$ be a monotone decreasing sequence bounded above by $R$ and converging to $\epsilon_V$.
Choose transversely non-degenerate Hamiltonians inductively by requiring that $H^{\tau_i}$ satisfying the following conditions.
\begin{enumerate}
\item $H^{\tau_i}\big|_V = \cg{H}\big|_V$.  To simplify later computations, assume $\cg{H}\big|_{\dd V} = 0$.
\item There exists $h^{\tau_i}:\RR\longrightarrow\RR$ such that $H^{\tau_i}(r, x) = h^{\tau_i}(e^r)$ on $[0, R)\times\dd_-W$.
\item $H^{\tau_i}$ is convex on $(\epsilon_i, R)\times\dd_-W$ and concave on $(0, \epsilon_i)\times\dd_-W$ (adjust $\cg{H}$ if necessary).
\item $h^{\tau_i}$ is linear of slope $\tau_i$ on $\dd_-W\times[\epsilon_i - \epsilon_V, \epsilon_i]$,
\item After shifting by a constant, $H^{\tau_i}\big|_{\left((\epsilon_{i+1}, R)\times\dd_-W\right)\cup W_+} = H^{\tau_{i+1}}\big|_{\left((\epsilon_{i+1}, R)\times\dd_-W\right)\cup W_+}$.  In particular, $H^{\tau_i}\big|_{W_+} = \cg{H}\big|_{W_+}$.
\item $H^{\tau_{i+1}}\geq H^{\tau_i}$ everywhere.
\end{enumerate}
We denote the set of such Hamiltonians by $Ad_-(V, M)$ and let $Ad(V, M) = Ad_+(M)\sqcup Ad_-(V, M)$.  See Figure \ref{fig:cobord-func} for a cartoon.

\begin{remark}
Solutions of $\dot{x} = X_{H^{\tau_i}}(x)$ are partitioned by whether or not they live in \\$V\cup\left([0, \epsilon_V]\times\dd_-W\right)$.  We will see that conditions (2) -- (5) ensure that solutions living ''close to $V$'' form a subcomplex of the Floer cochain complex of $H^{\tau_i}$ and that this subcomplex computes the symplectic homology of $V$.  Conditions (5) and (6) enable continuation maps to respect these subcomplexes, and condition (1) bounds the action of constant orbits, so that they are all eventually accounted for under completion by action.
\end{remark}

\begin{figure}[htbp!]
\begin{center}
\includegraphics[scale=.5]{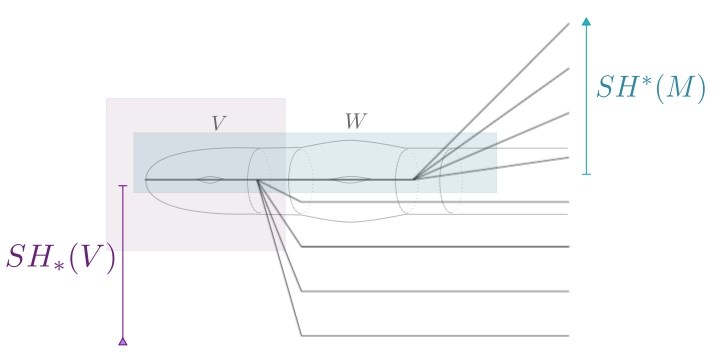}
\end{center}
\caption{The family of Hamiltonians $Ad(V, M)$ used to define the completed symplectic cohomology of a Liouville cobordism $W$ with monotone filling $V$.}
\label{fig:cobord-func}
\end{figure}

To see these conditions in play, let
\[
CF^*_{V, (a, b)}(H^{\tau_i}; \Gamma) = \left\la T^{\alpha}\zeta_x\in CF^*_{(a, b)}(H^{\tau_i};\Gamma)\hspace{.25cm}\bigg|\hspace{.25cm} \zeta_x\in o_x; \hspace{.1cm} x \subset V\cup [0, \epsilon_V]\times\dd_-W\right\ra.
\]
\begin{lemma}
For $H^{\tau_i}\in Ad_-(V, M)$, the subset $CF^*_{V, (a, b)}(H^{\tau_i}; \Gamma)$ is a subcomplex of $CF^*_{(a, b)}(H^{\tau_i};\Gamma)$.
\label{lem:subcomplex}
\end{lemma}

\begin{proof}
We show that there are no solutions of Floer's equation (\ref{eq:floer}) with positive limit a periodic orbit in $V\cup [0, \epsilon_V]\times\dd_-W$ and either

\begin{enumerate}
\item negative limit a non-constant orbit in $W_+\cup(\epsilon_V, R]\times\dd_-W$, or
\item negative limit a constant orbit in $W_+\cup(\epsilon_V, R]\times\dd_-W$.
\end{enumerate}
Assume for contradiction that $u(s, t)$ is a solution of Floer's equation with positive end an orbit $x(t)$ in $V\cup [0, \epsilon_V]\times\dd_-W$ and negative end an orbit $y(t)$ in $(\epsilon_V, R]\times\dd_-W\cup W_+.$
\begin{enumerate}
\item This can be found in \cite{cieliebak-o}.  Assume that $y(t)$ is a non-constant orbit.  By the construction of $H^{\tau_i}$, $y(t)\subset (\epsilon_i, R)\times \dd_-W$.  Since $h(r)$ is convex in this region, the proof of Proposition 5 in \cite{bourgeois-o} shows that $u(s, t)$ ``rises above'' $y(t).$  In other words, if $y(t)\subset$ $\{r_1\}\times \dd_-W,$ there exists $(s_1, t_1)\in\RR\times S^1$ and $r_2\in(r_1, R)$ such that $u(s_1, t_1)\subset\{r_2\}\times \dd_-W$.  The integrated maximum principle then applies to reach a contradiction.  We recall this final argument, which we learned from \cite{abouzaid}.

Let $\rho:[-R, R]\times\dd_-W\longrightarrow\RR$ be projection onto the $r$-coordinate, and choose $r_3\in(r_1, r_2)$ so that $r_3$ is a regular value of $\rho\circ u$.  
Consider the surface 
\[
\Sigma = u^{-1}\left([r_3, R]\times\dd_-W\cup W_+\right).
\]
Define $v:\Sigma\longrightarrow\widetilde{M}$ by $v = u\big|_{\Sigma}$.  As $v$ is a solution to Floer's equation (\ref{eq:floer}), the energy defined in equation (\ref{eq: energy}) may be rewritten as
\begin{align*}
E^{top}(v) &= \frac{1}{2}\int_{\Sigma} \omega(\dd_sv, J\dd_sv) + \omega\left(\dd_tv - X_{H^{\tau_i}}, J(\dd_tv - X_{H^{\tau_i}})\right)ds\wedge dt \\
&= \frac{1}{2}\int_{\Sigma} \omega(\dd_sv, \dd_tv - X_{H^{\tau_i}}) + \omega(\dd_tv - X_{H^{\tau_i}}, -\dd_s)ds\wedge dt\\
&= \frac{1}{2}\int_{\Sigma}2\omega(\dd_sv, \dd_tv) - 2\omega(\dd_s v, X_{H^{\tau_i}})ds\wedge dt\\
&= \int_{\Sigma} v^*\omega - v^*dH^{\tau_i}\otimes dt.
\end{align*}
By Stokes theorem this is equivalent to
\begin{equation}
E^{top}(v) = \int_{\dd\Sigma} v^*\lambda - H^{\tau_i}(v(t))dt.
\label{eq:stoke}
\end{equation}
As $\Sigma$ is a collection of bounded regions of $\CC^*$, 
$$\int_{\dd\Sigma} dt = 0,$$
 so that, in particular, 
$$\int_{\dd\Sigma}h^{\tau_i}(e^{r})dt = h^{\tau_i}(e^{r_3})\int_{\dd\Sigma}dt = 0$$
 and 
$$\int_{\dd\Sigma}\lambda_r(X_{H^{\tau_i}})dt = \int_{\dd\Sigma}e^{r}(h^{\tau_i})'(e^{r})dt = e^{r_3}(h^{\tau_i})'(e^{r_3})\int_{\dd\Sigma}dt = 0.$$
Thus, trivially, 
$$\int_{\dd\Sigma} H^{\tau_i}(v(t))dt = \int_{\dd\Sigma}\lambda(X_{H^{\tau_i}})dt.$$  
Using this equality, rewrite the energy as
\begin{align}
E^{top}(v) &= \int_{\dd\Sigma} v^*\lambda - \lambda(X_{H^{\tau_i}})\otimes dt\\
&= \int_{\dd\Sigma} \lambda(dv - X_{H^{\tau_i}}\otimes dt).
\end{align}
A solution $v(s, t)$ of Floer's equation satisfies $(dv - X_{H^{\tau_i}}\otimes dt)^{(0, 1)} = 0.$  As $J$ is conical and $H^{\tau_i}$ is radially-dependent, $JX_{H^{\tau_i}}$ is proportional to $\dd_r$ on $\{r_3\}\times\dd_-W$.  Thus, $\lambda\big|_{\{r_3\}\times\dd_-W}$ vanishes on $JX_{H^{\tau_i}}\big|_{\{r_3\}\times\dd_-W}$.  These observations imply that equation (6) can be written as  
\begin{align*}
E^{top}(v) &= \int_{\dd\Sigma} -\lambda J(dv - X_{H^{\tau_i}}\otimes dt)j\\
&= \int_{\dd\Sigma} -e^rdr\circ dv\circ j\\
&= \int_{\dd\Sigma} -e^rd(r\circ v)\circ j.
\end{align*}
A properly-oriented boundary vector $\zeta$ on $\dd\Sigma$ implies that $j\zeta$ points inwards.  Since $r\circ v$ achieves its minimum on $\dd\Sigma,$ $d(r\circ v)(j\zeta) \geq 0.$  The energy thus satisfies
\[E(v) \leq 0.\]
However, by definition, $E(v) \geq 0$, and so $E(v) = 0$.  Unpacking the properties of Floer solutions, this condition is only satisfied if $v$ is constant in $s$.  We reach a contradiction: $v$ cannot, in fact, exist.

\item 
Assume that $y(t)$ is a constant orbit.  Thus, $y\in W_+$.  Assume without loss of generality that $\epsilon_i$ is a regular value of $\rho\circ u$, and let $\Sigma = u^{-1}\left([\epsilon_i, R]\times\dd_-W\cup W_+\right).$  Note that $H((\epsilon_i, x)) = \lambda_{(\epsilon_i, x)}(X_{H^{\tau_i}}) + \sigma$ for some constant $\sigma > 0$.  While Equation (\ref{eq:stoke}) still holds, $\Sigma$ now decomposes a priori as a collection of bounded regions in $\CC^*$ and one unbounded region, which we call $\dd_+\Sigma$.  The previous computation shows that, in fact, the bounded regions do not exist.  Choose $\mathfrak{s}\in\RR$ such that $u\big|_{\{\mathfrak{s}\}\times S^1}\in\Sigma$ and $|\lambda(u(\mathfrak{s}))| < \sigma$.  The latter condition is possible because $y(t)$ is a constant orbit to which the curves $u(s, \cdot)$ converge smoothly, and so $\lim\limits_{s\rightarrow-\infty}\lambda(u(s)) = \lambda(y) = 0$.  The curves $\dd_+\Sigma$ and $\{\mathfrak{s}\}\times S^1$ bound a region in $\RR\times S^1$, which we call $\Sigma'$.  Let $v = u\big|_{\Sigma'}$.  Note that the boundary orientation of $\{\mathfrak{s}\}\times S^1$ in $\Sigma'$ is induced by $-dt$, so that
\[
\int_{\dd_+\Sigma}dt = \int_{\{\mathfrak{s}\}\times S^1}dt = 1.
\]
By assumption, $u\big|_{\{\mathfrak{s}\}\times S^1}\subset [\epsilon_i, R]\times\dd_+W\cup W_+$, a region on which $H^{\tau_i}$ is negative.  Thus, $\max_{t\in S^1}H(u(\mathfrak{s}, t)) < 0$.  Applying a computation similar to the computation above, we find that
\begin{align*}
0 \leq E(v) &= \int_{\dd_+\Sigma}-e^{\epsilon_i}d(r\circ v)\circ j - \int_{\dd_+\Sigma}\sigma dt -\int_{\{\mathfrak{s}\}\times S^1}v^*\lambda + \int_{\{\mathfrak{s}\}\times S^1}v^*H^{\tau_i}dt \\
&< -\sigma + \sigma + \max_{t\in S^1}H(u(\mathfrak{s}, t))\\
&< 0.
\end{align*}
A contradiction is again reached.
\end{enumerate}

\end{proof}

We have shown that $CF^*_{V, (a, b)}(H^{\tau_i}; \Gamma)$ is a subcomplex of $CF^*_{(a, b)}(H^{\tau_i};\Gamma)$, but, recalling the definition of symplectic chains, we actually want to find continuation maps $\{c^i\}$ so that
\[
\prod_{i=-1}^{-\infty}CF^*_{V, (a, b)}(H^{\tau_i}; \Gamma)[{\bm \theta}]
\hspace{1cm}
\text{is a subcomplex of}
\hspace{1cm}
\prod_{i=-1}^{-\infty}CF^*_{(a, b)}(H^{\tau_i};\Gamma)[{\bm \theta}]
\]
when equipped with the differential $\delta^*$ of equation (\ref{eq:diff}).

Due to condition (6) on elements of $Ad(V, M)$, there exists a constant $\kappa_i > 0$ such that  
\[
H^{\tau_i}\big|_{\left(\dd_-W\times(\epsilon_{i+1}, R)\right)\cup W_+} + \kappa_i= H^{\tau_{i+1}}\big|_{\left(\dd_-W\times(\epsilon_{i+1}, R)\right)\cup W_+}
\]
Let $\chi(s)$ be a bump function that is 1 when $s$ is very negative and 0 when $s$ is very positive.  Let $\{H_s:\widetilde{M}\times\RR\longrightarrow\RR\}_{s\in\RR}$ be an $\RR$-family of Hamiltonians, monotone decreasing in $s$, such that
\[
H_s\big|_{\left((\epsilon_{i+1}, R)\times\dd_-W\right)\cup W_+} = H^{\tau_i}\big|_{\left((\epsilon_{i+1}, R)\times\dd_-W\right)\cup W_+} + \kappa_i\cdot\chi(s).
\]
After choosing a suitable almost-complex structure, $H_s$ induces a continuation map 
\[
c^i: CF^*_{(a, b)}(H^{\tau_i};\Gamma)\longrightarrow CF^*_{(a, b)}(H^{\tau_{i+1}}; \Gamma).
\]

\begin{lemma} The continuation map $c^i$ restricts to a map $c^i: CF^*_{V, (a, b)}(H^{\tau_i}; \Gamma)\longrightarrow CF^*_{V, (a, b)}(H^{\tau_{i+1}}; \Gamma)$.
\label{lem:continuation}.
\end{lemma}

\begin{proof}
Let $y\subset (\epsilon_V, R]\times\dd_-W\cup W_+$ be a one-periodic orbit of $H^{\tau_{i+1}}$.  Choose a neighborhood $U$ of $y$ inside $(\epsilon_{i+1}, R)\times\dd_-W\cup W_+$.  By construction, $X_{H_s}$ is independent of $s$ on $U$.  The proof of Lemma \ref{lem:subcomplex} now applies verbatim, being only concerned with the behavior of trajectories in the part of $\widetilde{M}$ on which $H_s$ is $s$-independent.

\end{proof}

Define
\[
\widehat{SC_*^{V, (a, b)}}(M; \Gamma) := \prod_{i= -1}^{-\infty} CF^*_{V, (a, b)}(H^{\tau_i}; \Gamma)[{\bm \theta}].
\]
and denote the action-completion of $\widehat{SC_*^{V, (a, b)}}(M; \Gamma)$ by $\widehat{SC_*^V}(M; \Lambda)$.
\begin{lemma}
There is a chain isomorphism
\[
\widehat{SC_*}(V; \Lambda)\cong\widehat{SC_*^V}(M; \Lambda).
\]
\label{lem:chainiso}
\end{lemma}

\begin{proof}
Abuse notation slightly, and let $CF^*_{(a, b)}(H^{\tau_i}\big|_V; \Gamma)$ be the Floer complex associated to the restricted Hamiltonian $H^{\tau_i}:V\cup [0, \epsilon_V)\times\dd V\longrightarrow\RR$.  The elements of $CF^*_{(a, b)}(H^{\tau_i}\big|_V; \Gamma)$ are in clear bijection with the elements of $CF^*_{V, (a, b)}(H^{\tau_i}; \Gamma)$.  Furthermore, by the proof of Lemma \ref{lem:subcomplex}, the differentials are canonically identified.  Similarly, continuation maps are canonically identified, yielding a chain isomorphism
\[
\prod_{i=-1}^{-\infty} CF^*_{(a, b)}(H^{\tau_i}\big|_V; \Gamma)[{\bm \theta}] \cong \prod_{i=-1}^{-\infty} CF^*_{V, (a, b)}(H^{\tau_i}; \Gamma)[{\bm \theta}]. 
\]
After taking action limits, the left-hand side agrees with the completed symplectic chain complex of $V$.

\end{proof}

Choose a continuation map $c: CF^*_{(a, b)}(H^{\tau_{-1}}; \Gamma)\longrightarrow CF^*_{(a, b)}(H^{\tau_0}; \Gamma)$.  This choice induces a chain map $c: CF^*_{V, (a, b)}(H^{\tau_{-1}}; \Gamma)\longrightarrow CF^*_{(a, b)}(H^{\tau_0}; \Gamma)$.  Define a map $\frak{c}$ by the commutative diagram
\begin{equation}
\label{eq:continuation}
\begin{tikzcd}
SC_*^{V, (a, b)}(M; \Gamma) \arrow{r}{\mathfrak{c}} \arrow[twoheadrightarrow]{d}{\pi} & SC^*_{(a, b)}(M; \Gamma) \\ CF^*_{V, (a, b)}(H^{\tau_{-1}}; \Gamma){\bm \theta} \arrow{r}{c} & CF^*_{(a, b)}(H^{\tau_0}; \Gamma) \arrow[hookrightarrow]{u}{\iota}
\end{tikzcd}
\end{equation}
As $\dd^{fl}$ commutes with continuation maps, $\mathfrak{c}$ is a chain map.
\begin{definition}
The $(a, b)$-truncated symplectic cochain complex of $W$ is the cone of $\mathfrak{c}$.  Denote it by
\[
SC^*_{(a, b)}(W; \Gamma) := Cone\left(\mathfrak{c}:SC_*^{V, (a, b)}(M; \Gamma)\longrightarrow SC^*_{(a, b)}(M; \Gamma)\right).
\]
\end{definition}
\begin{definition}
The completed symplectic cochain complex of $W$ is
\[
\widehat{SC^*}(W; \Lambda) := \lim_{\substack{\longrightarrow \\ a}}\lim_{\substack{\longleftarrow \\ b}}SC^*_{(a, b)}(W; \Gamma)
\]
The homology of this complex is denoted by $\widehat{SH^*}(W; \Lambda)$.
\end{definition}

Analogously to the computations in Section \ref{rfh}, there is a long exact sequence
\begin{equation}
...\longrightarrow \widehat{SH_n}(V; \Lambda)\xlongrightarrow{\mathfrak{i}} \widehat{SH^n}(M; \Lambda)\xlongrightarrow{q} \widehat{SH^n}(W; \Lambda)\longrightarrow \widehat{SH_{n+1}}(V; \Lambda)\longrightarrow...
\label{diag:les-cob}
\end{equation}
which shows Theorem \ref{thm:les}.

\begin{remark}
As $M$ is monotone, the symplectic chain and cochain complexes are well-defined without truncating each Floer complex by action.  Denote these complexes by $SC_*(V;\Gamma)$ and $SC^*(M;\Gamma)$, respectively.  The map $\mathfrak{c}$ is also well-defined without truncating by action; call the cone of $\mathfrak{c}$ the symplectic cochain complex of $W$, denoted by $SC^*(W; \Gamma)$.  These three ''uncompleted'' complexes form a triangle analogous to (\ref{diag:exact-trunc}), and taking homology results in a long exact sequence analogous to (\ref{diag:les-cob}).
\end{remark}

\section{A non-vanishing theorem}
\label{non-vanish}

Computing symplectic cohomology is quite difficult; it has only been computed (in the monotone case) for negative line bundles by Ritter in \cite{ritter1}.  An easier line of inquiry is to ask, ``is symplectic cohomology non-zero?''  One method of answering this question affirmatively is to find a Lagrangian submanifold $L\subset M$ with non-vanishing Floer homology and show that there exists a map of unital rings $SH^*(M)\longrightarrow HF^*(L)$ from the symplectic cohomology of $M$ to the Floer homology of $L$.

For example, equation (6.4) of \cite{ritter-s} says that a monotone Lagrangian $L$ contained in a monotone manifold $M$ admits a map of unital rings
\[
SH^*(M; \Lambda)\longrightarrow HF^*(L; \Lambda).
\]
We will show that if, under suitable conditions, a Lagrangian $L$ is contained in the Liouville cobordism $W\subset M$, then this map factors through $\widehat{SH^*}(W; \Lambda)$ via the map $q:\widehat{SH^*}(M; \Lambda)\longrightarrow \widehat{SH^*}(W; \Lambda)$ appearing in the long-exact sequence (\ref{diag:les-cob}).
\[\begin{tikzcd}[column sep=scriptsize]
\widehat{SH^*}(M; \Lambda) \arrow[dr] \arrow[rr]
& & HF^*(L; \Lambda) \\
& Im(q) \arrow[ur] \arrow[draw = none]{d}[sloped, auto=false]{\subset}\\
& \widehat{SH^*}(W; \Lambda)
\end{tikzcd}\]
From this we will deduce the following theorem.

\nonvanish*

The conditions on $L$, orientability and monotonicity, control the behavior of Maslov discs.  Recall that a Lagrangian $L\subset M$ is {\it monotone} if the area and the Maslov index of any $J$-holomorphic disc with boundary on $L$ are positively proportional.  That is, there exists a constant $c > 0$ associated to $L$ such that for every $J$-holomorphic map $u:(D^2, \dd D^2)\longrightarrow (M, L)$ the symplectic area of $u$ and the Maslov index $\mu(u)$ satisfy
\begin{equation}
\label{eq:monotone-L}
\mu(u) = 2c\int_{D^2}u^*\omega.
\end{equation}
If $L$ is also orientable then the Maslov index of any such non-constant disc is at least two.

\subsection{Lagrangian quantum cohomology}
\label{lqc}

Fix a coefficient field $\KK$.  The Lagrangian Floer cohomology of a monotone Lagrangian submanifold $L$ with coefficients in a flat line bundle $E_{\gamma}$ is isomorphic to the Lagrangian quantum cohomology of $L$ with coefficients twisted by $\gamma\in H^1(L)$, where the holonomies of $E_{\gamma}$ are determined by $\gamma$.  (This is stated in Section 2.4 of \cite{biran-c1} and worked out in detail in \cite{biran-c3} in the untwisted case.)  We recall the definition of Lagrangian quantum cohomology.

Define a valuation on $\Lambda$ by
\begin{align}
val:\Lambda&\longrightarrow\RR\cup\{\infty\} \label{val}\\
\sum_{n=1}^{\infty} c_nT^{k_n} &= \left\{\begin{array}{cc} \min\limits_{c_n\neq 0}k_n & \exists\hspace{.1cm} c_n\neq 0 \\ \infty & \text{else}\end{array}\right.
\end{align}

Let $U_{\Lambda} = val^{-1}(0)$, and fix $\gamma \in H^1(L, U_{\Lambda})$.  Fix a Morse-Smale pair $(f, g)$ on $L$ and a generic almost-complex structure $J$ on $M$.

Let $\Phi_t$ be the flow of $-\nabla_g(f)$.  For critical points $x$ and $y$ of $f$ and an integer $\ell \geq 1$, let $\cg{M}_{\ell}(x, y; f, g, J)$ be the moduli space of tuples $(u_1, ..., u_{\ell})$, where
\begin{enumerate}
\item $u_i:(D^2, \dd D^2)\longrightarrow (M, L)$ is a non-constant $J$-holomorphic disc for all $1\leq i \leq \ell$,
\item for every $1 \leq i < \ell$ there exists $-\infty < t < 0$ such that $\Phi_t(u_{i+1}(1)) = u_i(-1)$, and
\item $u_1(1)$ lies in the unstable manifold of $x$ and $u_{\ell}(-1)$ lies in the stable manifold of $y$.
\end{enumerate}
Let $Aut(D^2, \pm 1)$ be the automorphisms of the disc fixing $-1$ and $1$, so that $Aut(D^2, \pm 1)^{\ell}$ acts on $\cg{M}_{\ell}(x, y; f, g, J)$.  Let $\cg{M}_0(x, y; f, g)$ be the moduli space of gradient flow lines of $f$ with negative asymptotic limit $x$ and positive asymptotic limit $y$.  $\RR$ acts on elements of $\cg{M}_0(x, y; f, g)$ by translation.  Denote by $\cg{M}^0(x, y; f, g, J)$ the rigid elements of 
\[
\bigslant{\cg{M}_0(x, y; f, g)}{\RR}\cup\bigcup_{\ell\geq 1}\bigslant{\cg{M}_{\ell}(x, y; f, g, J)}{Aut(D^2, \pm 1)^{\ell }}.
\]

\begin{remark}
\label{rmk:quantum-trans}
Transversality of the moduli spaces $\cg{M}_{\ell}(x, y; f, g, J)$ for generic triples $(f, g, J)$ is not automatic.  The discs may not be simple, or they may not be absolutely distinct.  However, Biran-Cornea showed in \cite{biran-c2} that somewhere-injectivity does not fail for dimension 0 and 1 strata of $\bigslant{\cg{M}_{\ell}(x, y; f, g, J)}{Aut(D^2, \pm 1)^{\ell }}$.  We can therefore use the moduli spaces $\cg{M}^0(x, y; f, g, J)$ to define a Floer homology theory, invariant up to generic choice of data $(f, g, J)$.
\end{remark}

Define a chain complex
\[
CF^*(L, E_{\gamma}) := \bigoplus_{x\in\Crit(f)}\Lambda\cdot x.
\]
A $\ZZ/2\ZZ$-grading on critical points is given by the Morse index, and we grade $T$ by $|T| = 0$.  The differential $\dd$ is given by
\[
\dd(y) = \sum_{\substack{x \in\Crit(f) \\ {\bm u} = (u_1, ..., u_{\ell}) \\ \in\cg{M}^0(x, y; f, g, J)}} \pm T^{\omega([u_1] + ... + [u_{\ell}])}\la\gamma, [\dd u_1] + ... + [\dd u_{\ell}]\ra\cdot x.
\]
The differential counts weighted `pearly trajectories' between $x$ and $y$ (see Figure \ref{fig:pearly-diff1}).  The sign is determined by a choice of orientations on the unstable manifolds of critical points of $f$ and a choice of spin structure on $L$.  (See the Appendix of \cite{biran-c1} for a careful discussion of orientations, in particular Section A.2.)  As shown in \cite{biran-c1}, $\dd$ is well-defined and squares to zero.  The homology of $CF^*(L, E_{\gamma})$ is the Lagrangian Floer homology $HF^*(L, E_{\gamma})$.

\begin{figure}[htbp!]
\center
\includegraphics[scale=.7]{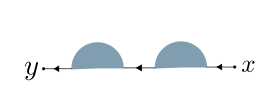}
\caption{The Lagrangian quantum differential}
\label{fig:pearly-diff1}
\end{figure}

Define the action $\cg{A}$ of an element $T^kx$, where $x\in\Crit(f)$ and $k\in\RR$, to be $\cg{A}(T^kx) = k$.  As $J$-holomorphic discs have non-negative area, the quantum differential increases $\cg{A}$.  We may thus consider the subcomplex
\[
CF^*_a(L, E_{\gamma}) := \left\la T^kx\hspace{.1cm}\big|\hspace{.1cm} k\in\RR,\hspace{.05cm}, x\in\Crit(f),\hspace{.05cm} \cg{A}(T^kx) > a\right\ra,
\]
with cohomology denoted by $HF^*_a(L, E_{\gamma})$, and the quotient complex
\[
CF^*_{(a, b)}(L, E_{\gamma}) := \bigslant{CF^*_a(L, E_{\gamma})}{CF^*_b(L, E_{\gamma})},
\]
with cohomology denoted by $HF^*_{(a, b)}(L, E_{\gamma})$.

Let $\Lambda_{>0} = val^{-1}\left((0, \infty]\right)$, where $val$ is as defined in (\ref{val}).
As we will be working with action-truncated complexes, define
\[
\Lambda_a = T^a\Lambda_{>0}
\]
and let
\[
\Lambda_{(a, b)} = \bigslant{\Lambda_a}{\Lambda_b}.
\]
Note that, by definition, $\Lambda = \lim\limits_{\substack{\longrightarrow \\ a}}\lim\limits_{\substack{\longleftarrow \\b}}\Lambda_{(a, b)}$, and
\begin{equation}
CF^*_{(a, b)}(L, E_{\gamma}) = \bigoplus_{x\in\Crit(f)} \Lambda_{(a, b)}\cdot x
\label{eq:finitesum}
\end{equation}

\begin{lemma}
\label{lem:equiv-quantum}
\[
\lim\limits_{\substack{\longrightarrow \\ a}}\lim\limits_{\substack{\longleftarrow \\b}}HF^*_{(a, b)}(L, E_{\gamma})\cong HF^*(L, E_{\gamma}) \cong H^*\left(\lim\limits_{\substack{\longrightarrow \\ a}}\lim\limits_{\substack{\longleftarrow \\b}}CF^*_{(a, b)}(L, E_{\gamma})\right).
\]
\end{lemma}

\begin{proof}
The first isomorphism is a canonical identification, analogous to Remark \ref{rmk:novikov}.  To see the second isomorphism, note that the right-hand sum of (\ref{eq:finitesum}) is finite and so commutes with both inverse and direct limits.  Thus,
\begin{align}
\label{eq:lim}\lim\limits_{\substack{\longrightarrow \\ a}}\lim\limits_{\substack{\longleftarrow \\b}}CF^*_{(a, b)}(L, E_{\gamma})
&= \lim\limits_{\substack{\longrightarrow \\ a}}\lim\limits_{\substack{\longleftarrow \\b}}\bigoplus_{x\in\Crit(f)} \Lambda_{(a, b)}\cdot x\\
&\cong \bigoplus_{x\in\Crit(f)} \lim\limits_{\substack{\longrightarrow \\ a}}\lim\limits_{\substack{\longleftarrow \\b}}\Lambda_{(a, b)}\cdot x\\
&\cong  \bigoplus_{x\in\Crit(f)}\Lambda\cdot x\\
&\cong CF^*(L, E_{\gamma}).
\end{align}
As the quantum differential is $T$-linear, the result follows.
\end{proof}

\begin{remark}
As the Lagrangian Floer cohomology of $L$, $HF^*(L, E_{\gamma})$ is a unital ring.  Suppose $HF^*(L, E_{\gamma})\neq 0$.  If $f$ has a unique minimum $\mathfrak{m}$, then $\mathfrak{m}$ represents the unit, and therefore survives in cohomology.
\label{rmk:quantum-neq0}
\end{remark}

\subsection{Proof of Theorem \ref{non-vanish}}

We first define a map $SC^*_{(a, b)}(M; \Gamma)\longrightarrow CF^*_{(a, b)}(L, E_{\gamma}).$  This map will count 'half-cylinder' solutions to Floer's equation that rise asymptotically to generators of some $CF^*_{(a, b)}(H^{\tau_i};\Gamma)$ and whose boundary lies in $L$.

To these ends, fix a Hamiltonian $H^{\tau_i}$.  Recall the function $\cg{H}$ built into the definition of $H^{\tau_i}$ (see Section \ref{cohomology}), and without loss of generality assume that $\cg{H} < 0$ in a neighborhood of $L$.  Let $\{H_s^{\tau_i}\}_{s\in[0, \infty)}$ be a one-parameter family of Hamiltonians such that $\lim\limits_{s\rightarrow\infty} H_s^{\tau_i} = H^{\tau_i}$, and $H^{\tau_i}_s(x) = 0$ when both $s$ is close to zero and $x$ is close to $L$.  Further assume that $H_s^{\tau_i}$ is monotone decreasing in $s$.  These conditions will ensure that we create a well-defined chain map.

For a periodic solution $x$ of $X_{H^{\tau_i}}$ and generic one-parameter family of cylindrical almost-complex structures $\{J_s\}_{s\in[0, \infty)}$, let $\cg{M}(x, L; H^{\tau_i})$ be the moduli space of maps $w:[0, \infty)\times S^1\longrightarrow M$ satisfying
\begin{enumerate}
\item $\lim\limits_{s\rightarrow\infty} w(s, t) = x(t)$,
\item \begin{minipage}{.9\textwidth}$w\big|_{\{0\}\times S^1}\in L$, and \end{minipage}
\begin{minipage}{.05\textwidth}($\star$)\end{minipage}
\item $\dd_s(w) + J_s(\dd_t(w) - X_{H_s^{\tau_i}}) = 0$.
\end{enumerate}

Fix a Morse-Smale pair $(f, g)$ on $L$ so that $f$ has a unique minimum $\mathfrak{m}$.  Let $\Phi_t$ be the flow of $-\nabla_g(f)$.  For each integer $\ell\geq 1$ and $p\in\Crit(f)$, let $\cg{M}_{\ell}(p, x; H^{\tau_i})$ be the moduli space of tuples $(u_1, ..., u_{\ell})$, where
\begin{enumerate}
\item $u_i:(D^2, \dd D^2)\longrightarrow (M, L)$ is a non-constant $J$-holomorphic disc for all $1\leq i \leq \ell - 1$,
\item $u_{\ell}\in\cg{M}(x, L; H^{\tau_i})$,
\item $u_1(1)$ is in the unstable manifold of $p$,
\item for every $1 \leq i < \ell$ there exists $-\infty < t < 0$ such that $\Phi_t(u_{i+1}(1)) = u_i(-1)$.
\end{enumerate}
Let $Aut(D^2, \pm 1)$ be the automorphisms of the disc fixing $-1$ and $1$, so that $Aut(D^2, \pm 1)^{\ell -1}$ acts on $\cg{M}_{\ell}(p, x; H^{\tau_i})$.  Let 
\[
\cg{M}(p, x; H^{\tau_i}) = \bigcup_{\ell\geq 1}\bigslant{\cg{M}_{\ell}(p, x; H^{\tau_i})}{Aut(D^2, \pm 1)^{\ell -1}}.
\]
Denote by $\cg{M}^0(p, x; H^{\tau_i})$ the rigid elements of $\cg{M}(p, x; H^{\tau_i})$.  Fix a class $\gamma\in H^1(L; U_{\Lambda})$.  Fix a spin structure on $L$ and Morse orientations, so that any element ${\bm u} \in\cg{M}^0(p, x; H^{\tau_i})$ determines a map $d_{{\bm u}}:o_x\longrightarrow\pm 1$.  Let
\begin{equation}
\label{eq:omega-u}
\omega({\bm u}) = \omega(-\tilde{x}\#u_{\ell}) + \omega(u_1) + ... + \omega(u_{\ell-1})
\end{equation}
and
\begin{equation}
\label{eq:bound-u}
[\dd{\bm u}] = [\dd u_1] + ... + [\dd u_{\ell -1}] + [u_{\ell}(0, \cdot)].
\end{equation}
Define a $T$-linear map
\begin{align*}
\iota^{\tau_i}_{(a, b)}: CF^*_{(a, b)}(H^{\tau_i};\Gamma)&\longrightarrow CF^*_{(a, b)}(L, E_{\gamma})\\
\zeta_x &\mapsto \sum_{p\in\Crit(f)}\sum_{\substack{{\bm u} = (u_1, ..., u_{\ell})\\\in\cg{M}^0(p, x; H^{\tau_i})}}d_{{\bm u}}(\zeta_x)T^{\omega({\bm u})}\la\gamma,[\dd{\bm u}]\ra \cdot p.
\end{align*}
Geometrically, we are mapping a cylinder $\tilde{x}\in M$ to the sum of cylinders $-\tilde{x}\#w$ with boundary on $L$ formed by gluing $\tilde{x}$ to the cylinder $w$, and adding in 'pearls' from pearly trajectories between $w(0, 1)$ and $p$.  See Figure \ref{fig:pearly-map1}.

\begin{figure}[htbp!]
\center
\includegraphics[scale=.7]{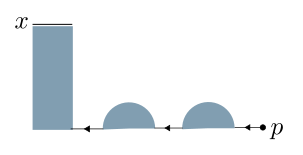}
\caption{The map $\iota^{\tau_i}$}
\label{fig:pearly-map1}
\end{figure}

\begin{lemma} $\iota^{\tau_i}_{(a, b)}$ is well-defined and descends to a map on homology.
\label{lemma:discs}
\end{lemma}
\begin{proof}
The techniques sketched to prove this Lemma appear in detail in \cite{biran-c1} and \cite{biran-c2}.

We first check that $\iota^{\tau_i}_{(a, b)}$ respects the action filtration.  A standard computation shows that, if $(u_1, ..., u_{\ell})\in\cg{M}^0(p, x; H^{\tau_i})$, then the energy of $u_{\ell}$ is
\begin{align*}
E(u_{\ell}) = -\omega(\tilde{x}\#-u_{\ell}) - \cg{A}_{H^{\tau_i}}(x) + \int_{\RR\times S^1}(\dd_sH_s^{\tau_i})(u_{\ell})dsdt. 
\end{align*}
Since the energy is always non-negative,
\begin{align*}
\cg{A}_{H^{\tau_i}}(x) - \int_{\RR\times S^1}(\dd_sH_s^{\tau_i})(u_{\ell})dsdt &\leq -\omega(\tilde{x}\#-u_{\ell}), \text{ so} \\
\cg{A}_{H^{\tau_i}}(x) &\leq \omega(-\tilde{x}\#u_{\ell})
\end{align*}
by the assumption that $H_s^{\tau_i}$ is monotone decreasing in $s$.  Finally, $\gamma\in H^1(L; U_{\Lambda})$ implies that $\la\gamma, w(0, \cdot)\ra\in U_{\Lambda}$, and each $u_i$ is $J$-holomorphic, so  $\omega(u_i) \geq 0$.  It follows that 
\[
T^{\omega({\bm u})}\la\gamma, w(0, \cdot)\ra\in \Lambda_{\omega(-\tilde{x}\#u_{\ell})}.
\]

Let $\cg{S}$ be a stratum of the set of maps $\{{\bm u} = (u_1, ..., u_{\ell}) \in\cg{M}(p, x; H^{\tau_i})\}$ of fixed homology class $A$: $[-\tilde{x}\#u_{\ell}]+[u_1]+...+[u_{\ell}-1]=A \in \bigslant{H_2(M, L\cup\beta)}{\ker(\omega)}$, where $\beta$ is the fixed representative in $[x]$ (see Subsection \ref{set-up}).  Assume that $\dim(\cg{S})\leq 1$, as these are the only strata contributing to the study of $\iota^{\tau_i}_{(a, b)}$.  The transversality for pearly trajectories proved in Section 3 of \cite{biran-c2} and the transversality for half-tubes discussed in \cite{albers} show that $\cg{S}$ is cut out transversely whenever the virtual dimension of $\cg{S}$ is less than or equal to $1$, and thus, by regularity, whenever $\dim(\cg{S})\leq 1$.  We therefore only need to show that bubbling does not contribute to compactification.  

There are six types of limit points that, a priori, contribute to the compactification of $\cg{S}$ (see Figure \ref{fig:bubbling}).
\begin{enumerate}[label=\alph*)]
\item cylinder breaking contributing to $\iota^{\tau_i}_{(a, b)}\circ\dd^{fl}$, \\
\item pearly-trajectory breaking contributing to $\dd\circ \iota^{\tau_i}_{(a, b)}$ \\
\item sphere-bubbling, \\
\item side-bubbling, where a disc bubbles off at a boundary point $u_i(q)$, where $q\neq \pm 1$ if $i< \ell$ and $q \neq 1$ if $i = \ell$, \\
\item disc bubbling at $q = u_i(\pm 1)$ (when $i < \ell$) or at $q = u_{\ell}(1)$, and \\
\item Morse-trajectory shrinking, where the trajectory between some $u_i$ and $u_{i+1}$ collapses, causing $u_i(-1)$ and $u_{i+1}(1)$ to collide.
\end{enumerate}

\begin{figure}
\centering
\includegraphics[scale=.5]{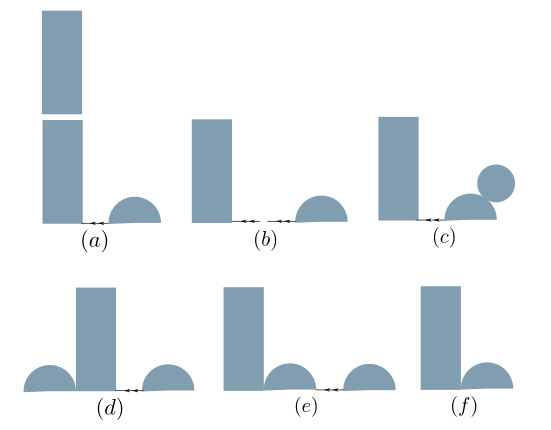}
\caption{Degenerations of $\cg{S}$}
\label{fig:bubbling}
\end{figure}
  
It thus suffices to show that, if $\cg{S}$ has dimension less than 2, the sum contribution of types (c), (d), (e), and (f) is zero.

We first tackle (c) and (d).  By monotonicity and orientability of $L$, the virtual dimension of side-bubbling or sphere-bubbling is at least 2.  If either occurs in a limit, then the other component of the limit is a stratum $\cg{S}'$ of $\cg{M}(p, x; H^{\tau_i})$ of virtual dimension two less than the virtual dimension of $\cg{S}$.  But by regularity this implies that $\dim(\cg{S}') \leq \dim(\cg{S}) - 2 < 0$.  This shows that (c) and (d) cannot occur.

There is a canonical bijection between elements of type (e) and elements of type (f).  An analysis of signs shows that an element of type (e) contributes with the opposite sign of its type (f) partner.  See Section A.2.1 in \cite{biran-c1} for a careful treatment of signs.  Thus, counting the limit points of both types (e) and types (f) yields zero.

If $\dim(\cg{S}) = 0$ then limits of types (a) and (b) cannot occur for index reasons, proving that $\iota^{\tau_i}_{(a, b)}$ is well-defined.

If $\dim(\cg{S}) = 1$ then the analysis of the boundary yields the equivalence 
\begin{equation}
\label{eq:chainmap}
\iota^{\tau_i}_{(a, b)}\circ \dd^{fl}(x) = \dd\circ\iota^{\tau_i}_{(a, b)}(x),
\end{equation}
as desired.

\end{proof}

Let $(\iota^{\tau_i}_{(a, b)})^*: HF^*_{(a, b)}(H^{\tau_i}; \Gamma)\longrightarrow HF^*_{(a, b)}(L, E_{\gamma})$ be the descent of $\iota^{\tau_i}_{(a, b)}$ to cohomology.  
\begin{lemma}
The collection of maps $\left\{(\iota^{\tau_i}_{(a, b)})^*\right\}_{i\in\NN}$ induces a map $\lim\limits_{\substack{\longrightarrow \\ i}} HF^*_{(a, b)}(H^{\tau_i}; \Gamma)\longrightarrow HF^*_{(a, b)}(L, E_{\gamma})$.
\label{lem:descent}
\end{lemma}

\begin{proof}
We must show that $(\iota^{\tau_{i+1}}_{(a, b)})^*\circ(c^i)^* =(\iota^{\tau_i}_{(a, b)})^*$.  We will do this by finding a chain homotopy $S$ such that
\[
\iota^{\tau_{i+1}}_{(a, b)}\circ c^i - \iota^{\tau_i}_{(a, b)} = S\circ\dd^{fl} + \dd\circ S.
\]
Let $x\in \cg{P}(H^{\tau_i})$.  Choose a regular homotopy $\{H_s\}_{s\in\RR}$ with $H_s = H^{\tau_i}$ when $s > 0$ and $H_s = H^{\tau_{i+1}}$ when $s < -1$.

Choose a generic smooth family of Hamiltonians $H:[0, \infty)\times\RR\times S^1\times M\longrightarrow \RR$ such that $H\big|_{\{\mathfrak{s}\}\times\{s\}\times S^1\times M}$ is equal to $H^{\tau_{i+1}}_s$ when $s < \mathfrak{s} - 1$ and $H^{\tau_i}_s$ when $s > \mathfrak{s}$.  Assume that, for $s\in[-1, 0]$,
\[
\lim\limits_{\mathfrak{s}\rightarrow\infty}\left(H(\mathfrak{s}, s+\mathfrak{s}, t, x) - H_s(t, x)\right) = 0.
\]

Choose a generic $[0, \infty)$-family of domain-dependent cylindrical almost-complex structures $J_{s, \mathfrak{s}}$.  Fix $p\in\Crit(f)$.  Let $\cg{M}^1(p, x, A; H)$ be the one-dimensional strata of the space of tuples $\left(\mathfrak{s}, (u_1, ..., u_{\ell})_{\mathfrak{s}}\right)$, where $\mathfrak{s}\in[0, \infty)$ and $(u_1, ..., u_{\ell})_{\mathfrak{s}}\in\cg{M}^0(p, x, A; H_{\mathfrak{s}})$.
Also require that $[-\tilde{x}\#u_{\ell}]_{\mathfrak{s}}+[u_1]_{\mathfrak{s}}+...+[u_{\ell}]_{\mathfrak{s}} = A$ is a fixed class in $\bigslant{H_2(M, L\cup\beta)}{\ker(\omega)}$ for every $\mathfrak{s}$, where $\beta$ is the fixed representative in $[x]$ (see Subsection \ref{set-up}).  This is a 1-dimensional manifold with a compactification given by cylinder-breaking and disc-bubbling.  A priori, the (0-dimensional) boundary components of the compactification take one of six forms.
\begin{enumerate}[label=\alph*)]
\item On the boundary $\mathfrak{s}=0$ appears the elements of the moduli space $\cg{M}^0(p, x, A; H^{\tau_i})$.  This corresponds to the part of $\iota^{\tau_i}_{(a, b)}(\zeta_x)$ that contributes terms of action $\omega(A)$, twisted by $\gamma$.
\item In the limit $\mathfrak{s}\longrightarrow\infty$ appear elements of the product
\[
\bigsqcup\limits_{\substack{z\in\cg{P}(H^{\tau_{i+1}}) \\ B\#C = A}}\cg{M}^0(p, z, B; H^{\tau_{i+1}})\times \cg{M}^0(z, x, C; H_s),
\]
where $\cg{M}^0(z, x, C; H_s)$ is the space of index-0 Floer solutions between $x$ and $z$ induced by the family of Hamiltonians $H_s$, and contributing terms in `$T^{\omega(C)}o_z$' to $c^i(\zeta_x)$.  This product contributes terms of action $\omega(A)$ to $\iota^{\tau_{i+1}}_{(a, b)}\circ c^i(\zeta_x)$, twisted by $\gamma$.
\item The moduli space can degenerate at an interior point $\mathfrak{s}\in(0, \infty)$, and near $s= \infty$, yielding elements of the form
\[
\bigsqcup\limits_{\substack{z\in\cg{P}(H^{\tau_{i}}) \\ \mu(z) - \mu(x) = 1 \\ B\#C = A}}\cg{M}^{-1}(p, z, B; H^{\tau_{i}}_{\mathfrak{s}})\times \cg{M}^0(z, x, C; H^{\tau_{i}}),
\]
where $\cg{M}^{-1}(p, z, B; H^{\tau_{i}}_{\mathfrak{s}})$ is the moduli space of rigid Floer/pearly trajectory amalgamates of virtual dimension $-1$ that can occur between $z$ and $p$ if $H_{\mathfrak{s}}$ is not regular (restricted, of course, to the relative homology class $B$).
\item The moduli space can degenerate at an interior point $\mathfrak{s}\in(0, \infty)$ and within a pearly trajectory, yielding elements of the form
\[
\bigsqcup\limits_{\substack{q\in\Crit(f) \\ B\#C = A \\ |p| - |q| - \mu(C) = 0 }}\cg{M}^0(p, q, C)\times\cg{M}^{-1}(q, x, B; H^{\tau_{i}}_{\mathfrak{s}}).
\]
\item Finally, bubbling may occur.  However, as in the proof of Lemma \ref{lemma:discs}, the contribution of disc and sphere bubbling is zero.
\end{enumerate}
Standard gluing techniques show that the degenerations of types (a) -- (d) do indeed appear.  In the notation of equations (\ref{eq:omega-u}) and (\ref{eq:bound-u}), define $S$ on generators by
\begin{align*}
S: \bigoplus_{i\in\NN}CF^*_{(a, b)}(H^{\tau_i}; \Gamma)&\longrightarrow CF^*_{(a, b)}(L, E_{\gamma}) \\
\zeta_x &\mapsto \sum_{\substack{\mathfrak{s}\in(0, \infty) \\ B\in H_2(M, L) \\ q\in\Crit(f)}}\sum_{{\bm u}\in\cg{M}^{-1}(q, x, B; H^{\tau_i}_{\mathfrak{s}})}d_{{\bm u}}(\zeta_x)T^{\omega({\bm u})}\la \gamma, [\dd{\bm u}]\ra\cdot q
\end{align*}
and extend $T$-linearly.

As the limiting degenerations at $\mathfrak{s} = \infty$ are regular it follows from Gromov compactness that there are finitely many degenerations of types (c), (d), and (e), and so $S$ is well-defined.

Counting boundary components of type (c) yields the ``$T^{\omega(A)}$'' component of $S\circ\dd^{fl}$ and counting boundary components of type (d) yields the ``$T^{\omega(A)}$'' component of $\dd\circ S$.  From this we deduce that
\[
\iota^{\tau_{i+1}}_{(a, b)}\circ c^i - \iota^{\tau_i}_{(a, b)} = S\circ\dd^{fl} + \dd\circ S.
\]
\end{proof}

There is a surjective map
\[
\Psi: \widehat{SH^*}(M; \Lambda)\longrightarrow \lim_{\substack{\longrightarrow \\ a}}\lim_{\substack{\longleftarrow \\ b}}\lim_{\substack{\longrightarrow \\ i}} HF^*_{(a, b)}(H^{\tau_i}; \Gamma),
\]
defined using the natural commutativity of direct limits with cohomology and the projection $H^*\left(\lim\limits_{\substack{\longleftarrow\\ b}}SC^*_{(a, b)}(M; \Gamma)\right)\longrightarrow\lim\limits_{\substack{\longleftarrow \\ b}} SH^*_{(a, b)}(M; \Gamma)$ that appears in the Milnor exact sequence
\[
0\longrightarrow\lim_{\substack{\longleftarrow \\ b}}{}^1SH^*_{(a, b)}(M)\longrightarrow H^*\left(\lim\limits_{\substack{\longleftarrow\\ b}}SC^*_{(a, b)}(M; \Gamma)\right)\longrightarrow\lim_{\substack{\longleftarrow \\ b}} SH^*_{(a, b)}(M; \Gamma)\longrightarrow 0.
\]
Under the equivalence $HF^*(L, E_{\gamma})\cong \lim\limits_{\substack{\longrightarrow \\ a}}\lim\limits_{\substack{\longleftarrow \\b}}HF^*_{(a, b)}(L, E_{\gamma})$ proved in Lemma \ref{lem:equiv-quantum}, let
\[
\widehat{\iota^*}:\lim_{\substack{\longrightarrow \\ a}}\lim_{\substack{\longleftarrow \\ b}}\lim_{\substack{\longrightarrow \\ i}} HF^*_{(a, b)}(H^{\tau_i}; \Gamma)\longrightarrow HF^*(L, E_{\gamma})
\]
be the induced map formed by taking limits over the maps $\left\{\iota^{\tau_i}_{(a, b)}\right\}_{i, a, b}$.  Define
\[
\widehat{\cg{I}^*} = \widehat{\iota^*}\circ\Psi: \widehat{SH^*}(M; \Lambda)\longrightarrow HF^*(L, E_{\gamma}).
\]

\begin{lemma}
The map $\widehat{\cg{I}^*}$ is non-vanishing.
\label{lem:nonvanish}
\end{lemma}

\begin{proof}
As $\Psi$ is a surjection, it suffices to prove that $\widehat{\iota^*}$ is non-vanishing.

Let $H_s$ be an $\RR$-family of Hamiltonians that is equal to $H^{\tau_0}$ when $s << 0$ and identically zero when $s >> 0$.  Choose a generic $\RR$-family of cylindrical almost-complex structures $J_s$.  Let $\cg{M}^0(x)$ be the set of rigid maps $w:\RR\times S^1 \cup\{\infty\}\longrightarrow M$ satisfying $\dd_sw + J_s(\dd_tw -X_{H_s}) = 0$, and such that $\lim\limits_{s\rightarrow-\infty}w(s, \cdot) = x(\cdot)$.  Each $w$ determines an element $\zeta_x^w\in o_x$.  Let
\[
Z = \sum\limits_{\substack{ x \in \cg{P}(H^{\tau_0}) \\ w \in \cg{M}^0(x)}} T^{\omega(\tilde{x}\#w)}\cdot \zeta_x^w \in CF^*(H^{\tau_0}; \Lambda).
\]
The usual analysis of the boundary of a dimension-one moduli space of curves shows that $Z$ is a well-defined cycle.

Recall that we chose the Morse function $f$ on $L$ to have a unique minimum $\mathfrak{m}$.  For $p\in\Crit(f)$, let $\cg{M}^{0}(p, A)$ be the space of rigid pearly trajectories $\{(u_1, ..., u_{\ell})\}_{\ell\geq 1}$, defined in the same way as $\cg{M}^0(p, x, A; H^{\tau_0})$, but where $u_{\ell}$ is also now a (possibly constant) $J$-holomorphic disc with one interior marked point.  If $u_{\ell}$ is not constant, the sequence $(u_1, ..., u_{\ell})$ is not rigid.  If $u_{\ell}$ is constant and $p\neq\mathfrak{m}$ then either
\begin{enumerate}
\item there is a non-constant gradient trajectory $\beta(t)$ with $\beta(0) = Im(u_{\ell})$, in which case sliding the image of $u_{\ell}$ along $\beta(t)$ shows that $(u_1, ..., u_{\ell})$ is not rigid, or
\item $Im(u_{\ell})\in\Crit(f)\setminus\mathfrak{m}$.  Then $Im(u_{\ell}) = p$, and $u_{\ell}$ sits inside a strata of $\cg{M}(p, A)$ containing the moduli space
\[
\{(u_1); u_1 \text{ is a constant disc mapping into the unstable manifold of } p\},
\]
which, as $\mathfrak{m}$ is the unique minimum, precludes the rigidity of $u_{\ell}$.
\end{enumerate}
Thus, 
\[
\bigsqcup\limits_{\substack{p\in\Crit(f)\\ A\in H_2(M, L)}}\cg{M}^{0}(p, A) = \{\mathfrak{m}\}.
\]

Define $\iota^{\tau_0}: CF^*(H^{\tau_0}; \Gamma)\longrightarrow CF^*(L, E_{\gamma})$ analogously to $\iota^{\tau_0}_{(a, b)}$, but without truncating by action.  As in the proof of Lemma \ref{lem:descent},
\[
\iota^{\tau_0}(Z) = \mathfrak{m} + \dd\circ S(Z)
\]
for some chain homotopy $S$, and so $(\iota^{\tau_0})^*([Z]) = [\mathfrak{m}]$.  

By assumption, $HF^*(L, E_{\gamma})$ is a non-zero unital ring with unit represented by $\mathfrak{m}$ (see Remark \ref{rmk:quantum-neq0}).  It follows that $(\iota^{\tau_0})^*(Z)\neq 0$.

As $Z$ is a cycle, $\iota^{\tau_0}$ descends to a non-trivial map $(\iota^{\tau_0})^*$ on homology.  We will show that the non-vanishing of this map implies Lemma \ref{lem:nonvanish}.

Let $\widehat{\iota^{\tau_0}} = \lim\limits_{\substack{\longrightarrow \\ a}}\lim\limits_{\substack{\longleftarrow \\b}}(\iota_{(a, b)}^{\tau_0})$.  Analogously to Lemma \ref{lem:equiv-quantum} there is a quasi-isomorphism 
\[
\lim\limits_{\substack{\longrightarrow \\ a}}\lim\limits_{\substack{\longleftarrow \\b}}HF^*_{(a, b)}(H^{\tau_0}; \Gamma) \cong HF^*(H^{\tau_0}; \Lambda).
\]
This isomorphism induces a commutative diagram
\[
\begin{tikzcd}
\lim\limits_{\substack{\longrightarrow \\ a}}\lim\limits_{\substack{\longleftarrow \\b}}HF^*_{(a, b)}(H^{\tau_0}; \Gamma) \arrow{d}{(\widehat{\iota^{\tau_0}})^*}\arrow{r}{\simeq} & HF^*(H^{\tau_0}; \Lambda) \arrow{d}{(\iota^{\tau_0})^*}\\
\lim\limits_{\substack{\longrightarrow \\ a}}\lim\limits_{\substack{\longleftarrow \\b}}HF^*_{(a, b)}(L, E_{\gamma})\arrow{r}{\simeq} & HF^*(L, E_{\gamma})
\end{tikzcd}
\]
It follows that $(\widehat{\iota^{\tau_0}})^*\neq 0$.

The inclusion $CF^*_{(a, b)}(H^{\tau_0}; \Gamma)\hookrightarrow \lim\limits_{\substack{\longrightarrow \\ i}} CF^*_{(a, b)}(H^{\tau_i}; \Gamma)$ induces a map 
\[
\lim\limits_{\substack{\longrightarrow \\ a}}\lim\limits_{\substack{\longleftarrow \\b}}HF^*_{(a, b)}(H^{\tau_0}; \Gamma)\longrightarrow \lim\limits_{\substack{\longrightarrow \\ a}}\lim\limits_{\substack{\longleftarrow \\b}}\lim\limits_{\substack{\longrightarrow \\ i}} HF^*_{(a, b)}(H^{\tau_i}; \Gamma).
\]
By the construction of $\widehat{\iota^{\tau_0}}$, the following diagram commutes
\[
\begin{tikzcd}
\lim\limits_{\substack{\longrightarrow \\ a}}\lim\limits_{\substack{\longleftarrow \\b}}HF^*_{(a, b)}(H^{\tau_0}; \Gamma) \arrow{dr}[left]{\widehat{(\iota^{\tau_0}})^*\hspace{10pt}} \arrow{r} & \lim\limits_{\substack{\longrightarrow \\ a}}\lim\limits_{\substack{\longleftarrow \\b}}\lim\limits_{\substack{\longrightarrow \\ i}} HF^*_{(a, b)}(H^{\tau_i}; \Gamma) \arrow{d}{\widehat{\iota^*}} \\
& HF^*(L, E_{\gamma})
\end{tikzcd}
\]
from which we deduce that $\widehat{\iota^*}$, and therefore $\widehat{\cg{I}^*}$, is non-vanishing.

\end{proof}

We are now in position to prove Theorem \ref{thm:nonvanish}.

\begin{proof}
The long-exact sequence (\ref{diag:les-cob}) induces a short-exact sequence
\[
0\longrightarrow \widehat{SH_*}(V; \Lambda)/Ker(\mathfrak{i}^*)\longrightarrow \widehat{SH^*}(M; \Lambda)\longrightarrow Im(q^*)\longrightarrow 0
\]
where $Im(q^*)\subset \widehat{SH^*}(W; \Lambda).$

We want to show that $\widehat{\cg{I}^*}$ factors through $Im(q^*).$  By the universal property of quotients, it suffices to show that $\widehat{\cg{I}^*}\circ \mathfrak{i}^*$ is zero on $\widehat{SH_*}(V; \Lambda).$  Fix $A\in\RR$.  The Mittag-Leffler condition is trivially satisfied for each inverse system $\{SC_*^{(A, b)}(V; \Gamma)\}_b$ and $\{\Lambda_{(A, b)}\}_b$ (the maps defining the inverse systems are surjective).  To each inverse system is therefore associated a Milnor $\lim^1$ short exact sequence.  By naturality of this sequence, and as direct limits preserve exactness, there is a commutative diagram of short exact sequences.
\[
\begin{tikzcd}
0 \arrow{r} &\lim\limits_{\substack{\longrightarrow \\ a}}\lim\limits_{\substack{\longleftarrow \\b}}{}^1SH_*^{(a, b)}(V; \Lambda) \arrow{r} \arrow{d}{\mathfrak{i}^*} & \widehat{SH_*}(V; \Lambda) \arrow{r} \arrow{d}{\mathfrak{i}^*}&\lim\limits_{\substack{\longrightarrow \\ a}}\lim\limits_{\substack{\longleftarrow \\b}}SH_*^{(a, b)}(V; \Lambda) \arrow{r} \arrow{d}{\mathfrak{i}^*} & 0 \\
0 \arrow{r} &\lim\limits_{\substack{\longrightarrow \\ a}}\lim\limits_{\substack{\longleftarrow \\b}}{}^1SH^*_{(a, b)}(M; \Lambda) \arrow{r} \arrow{d} & \widehat{SH^*}(M; \Lambda) \arrow{r} \arrow{d}{\widehat{\cg{I}^*}}&\lim\limits_{\substack{\longrightarrow \\ a}}\lim\limits_{\substack{\longleftarrow \\b}}SH^*_{(a, b)}(M; \Lambda) \arrow{r} \arrow{d}{\widehat{\iota^*}} & 0 \\
0 \arrow{r} & 0  \arrow{r} & HF^*(L, E_{\gamma}) \arrow{r}{\simeq} & \lim\limits_{\substack{\longrightarrow \\ a}}\lim\limits_{\substack{\longleftarrow \\b}}HF^*_{(a, b)}(L, E_{\gamma})\arrow{r} & 0
\end{tikzcd}
\]
It thus suffices to show that the map $\lim\limits_{\substack{\longrightarrow \\ a}}\lim\limits_{\substack{\longleftarrow \\b}}SH_*^{(a, b)}(V; \Lambda)\longrightarrow\lim\limits_{\substack{\longrightarrow \\ a}}\lim\limits_{\substack{\longleftarrow \\b}}HF^*_{(a, b)}(L, E_{\gamma})$ is zero.  In particular, it suffices to show that each map $SH_*^{(a, b)}(V; \Lambda)\longrightarrow HF^*_{(a, b)}(L, E_{\gamma})$ is zero.

Let ${\bf X}\in SH_*^{(a, b)}(V; \Lambda)$.  A representative cochain of ${\bf X}$ is of the form
\[
X:=\{\zeta_i\} + \{\eta_i\}{\bm \theta}\in\prod_{i=-1}^{-\infty} CF^*_{V, (a, b)}(H^{\tau_i}; \Gamma)[{\bm \theta}].
\]
We will show that $\left(\iota_{(a, b)}\circ \mathfrak{i} \right)(X) = 0$.

Define a map $\iota^{\tau_i}_{(a, b)}:CF^*_{(a, b)}(H^{\tau_i}; \Gamma)\longrightarrow CF^*_{(a, b)}(L, E_{\gamma})$ by extending the construction of the maps defined in Section \ref{non-vanish} to the Floer cochain complexes of negatively-indexed Hamiltonians.  We choose Hamiltonians $H_s^{\tau_i}$ to define each $\iota^{\tau_i}_{(a, b)}$ so that $X_{H_s^{\tau_i}}$ agrees with $X_{H_s^{\tau_0}}$ on $W_+\cup(\epsilon_{i+1}, R)\times\dd_-W$ for all $s$.  In particular, $H_s^{\tau_i}(x)$ is constant when both $x$ is close to $L$ and $s$ is close to $0$, and $|H_s^{\tau_i}(q) - H^{\tau_i}(q)|$ is small for all $s$ and $q\in L$. 

By definition,
\begin{equation}
\left(\iota_{(a, b)}\circ \mathfrak{i} \right)^*([X]) = (\iota_{(a, b)}\circ c^{-1})^*([\eta_{-1}]) = (\iota^{\tau_0}_{(a, b)}\circ c^{-1})^*([\eta_{-1}]) = (\iota_{(a, b)}^{\tau_{-1}})^*([\eta_{-1}]),
\label{eq:eta1-neq0}
\end{equation}
where the last equality follows from the equality
\begin{equation}
(\iota^{\tau_i}_{(a, b)}\circ c^{i-1})^* = (\iota^{\tau_{i-1}}_{(a, b)})^*
\label{eq:iota-cont}
\end{equation}
derived in the proof of Lemma \ref{lem:descent}.

The first term in the equation $0 = \delta^*(X)$ is
\[
0 = \left\{\dd^{fl}(\zeta_i) + c^{i-1}(\eta_{i-1}) - \eta_i)\right\}_{i = -1}^{-\infty}
\]
Composing with $\iota^{\tau_i}_{(a, b)}$ yields the componentwise equalities
\[
\iota^{\tau_i}_{(a, b)}(\eta_i) = \iota^{\tau_i}_{(a, b)}\circ\dd^{fl}(\zeta_i) + \iota^{\tau_i}_{(a, b)}\circ c^{i-1}(\eta_{i-1}).
\]
The proof of Lemma \ref{lemma:discs} shows that $(\iota^{\tau_i}_{(a, b)}\circ\dd^{fl})^* = 0$, and so
\begin{equation}
(\iota^{\tau_i}_{(a, b)})^*([\eta_i]) = (\iota^{\tau_i}_{(a, b)}\circ c^{i-1})^*([\eta_{i-1}]).
\label{eq:etai-neq0}
\end{equation}
for any $i$.  If $(\iota^{\tau_0}_{(a, b)}\circ c^{-1})^*([\eta_{-1}])\neq 0$, then by equations (\ref{eq:eta1-neq0}), (\ref{eq:iota-cont}), and (\ref{eq:etai-neq0}), 
\begin{equation}
(\iota^{\tau_{i}}_{(a, b)})^*([\eta_{i}])=(\iota^{\tau_{-1}}_{(a, b)})^*([\eta_{-1}])\neq 0.
\label{eq:eta-contradiction}
\end{equation}

Recall that action is decreased by Floer trajectories.  Thus, if $w(s, t)$ is a Floer solution of $X_{H^{\tau_i}}$ with positive asymptotic limit $x\in\cg{P}(H^{\tau_i})$, then for any $s\in\RR$
\begin{equation}
\label{action-est}
\cg{A}_{H^{\tau_i}}(x) \leq \omega(-\tilde{x}\#w\big|_{(s, \infty)\times S^1}) + \int_0^1H^{\tau_i}(w(s, t))dt.
\end{equation}
By assumption, the action $\cg{A}$ of every non-zero summand of $\iota^{\tau_i}(\eta_i)$ is in $(a, b)$.  As $\lim\limits_{i\rightarrow\infty}\tau_i = -\infty$, we may choose $i$ so that 
\[
\sup\limits_{q\in L}H_0^{\tau_i}(q) < a - b.
\]  

Write $\eta_i = \sum_{\ell = 1}^{m_i} \eta_i^{\ell}$, where $\eta_i^{\ell} = a_{\ell}T^{\alpha_{\ell}}\mathfrak{y}_i^{\ell}$ for some $\mathfrak{y}_i^{\ell}\in o_x; x\in\cg{P}(H^{\tau_i})$, $\alpha_{\ell}\in\RR$, and $a_{\ell}\in\KK^*$.  As the $J$-holomorphic discs in a pearly trajectory have non-negative area, and thus contribute non-negatively to action, (\ref{action-est}) implies that at least one of the $\mathfrak{y}_i^{\ell}$ satisfies
\[
\cg{A}_{H^{\tau_i}}(T^{\alpha_{\ell}}\mathfrak{y}_i^{\ell}) < b + \sup\limits_{q\in L}H_0^{\tau_i}(q) < b + a - b = a.
\]
This contradicts the assumption that $T^{\alpha_{\ell}}\mathfrak{y}_i^{\ell} \in CF^*_{(a, b)}(H^{\tau_i}; \Gamma)$.  It follows that 
\[
(\iota^{\tau_0}_{(a, b)}\circ c^{-1})^*([\eta_{-1}]) = (\iota^{\tau_i}_{(a, b)})^*([\eta_i]) = 0.
\]

\end{proof}

\begin{example}
The total space of the line bundle $\cg{O}(-k)\longrightarrow \CC P^m$ is monotone whenever $1 \leq k \leq m$, and contains a Lagrangian torus in the radius-$\frac{1}{\sqrt{\pi(1+m-k)}}$ sphere bundle that satisfies the conditions of Theorem \ref{thm:nonvanish} \cite{ritter-s}.  It follows that $\widehat{SH^*}(W; \Lambda)\neq 0$ for any Liouville cobordism $W$ containing the sphere bundle of radius $\frac{1}{\sqrt{\pi(1+m-k)}}$.
\end{example}

\section{Computations for some annulus bundles over $\CC P^1$}
\label{example}

The tautological bundle over $\CC P^1$ contains a Lagrangian torus in the radius--$\frac{1}{\sqrt{\pi}}$ sphere bundle that satisfies all of the conditions of Theorem \ref{thm:nonvanish} (see Figure \ref{fig:polytope}) \cite{smith}.  If $W$ is a Liouville cobordism between two sphere bundles, then Theorem \ref{thm:nonvanish} guarantees that $\widehat{SH^*}(W; \Lambda) \neq 0$ if $W$ contains the radius--$\frac{1}{\sqrt{\pi}}$ sphere bundle.  It transpires that the converse is true.

\examplee*

Let $M$ be a Liouville cobordism with empty negative boundary (i.e. a Liouville domain).  Cieliebak-Frauenfelder-Oancea showed in \cite{cieliebak-f-o} that the uncompleted symplectic cohomology of a trivial cobordism $W \subset \widetilde{M}$ containing $\dd M$ is isomorphic to the uncompleted Rabinowitz Floer homology of $\dd M$.  We expect a relationship in this flavor between $\widehat{RFH^*}(\Sigma; \Lambda)$ and the Rabinowitz Floer homology of a contact hypersurface $\Sigma$ in a negative line bundle $E$ studied by Albers-Kang \cite{albers-k}.  In particular, Albers-Kang showed that the Rabinowitz Floer homology of a sphere bundle in $E$ of radius less than $\frac{1}{\sqrt{\pi}}$ vanishes.  While they claim that this vanishing result extends to the sphere bundle of radius $\frac{1}{\sqrt{\pi}}$, Theorem \ref{prop:example} implies otherwise.

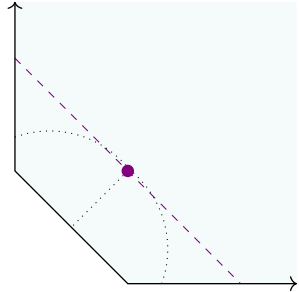
\begin{figure}[htbp!]
\center
\begin{tikzpicture}
\begin{tikzpicture}[scale=1.5]

    \draw [color=white, fill=teal!4] (1,0) -- (2.5,0) -- (2.5,2.5) -- (0,2.5) -- (0,1) -- cycle;

    \coordinate  (A) at (1,0);
    \coordinate (B) at (2.5,0);
    \coordinate (C) at (0,2.5);
    \coordinate (D) at (0,1);
    
    \coordinate (E) at (1, 1);
    \coordinate (F) at (.5, .5);
    \draw[dotted] (E) -- (F);
    \draw[dotted] (1.3, 0) arc (-20:55:.9);
    \draw[dotted] (0, 1.3) arc (110:35:.9);
    \draw (1, 1) node[circle, fill=violet, scale=.5] {};
    
    \draw[dashed, violet] (0, 2) -- (2, 0);

    \draw [<->, line width=.5pt] (B) -- (A) -- (D) -- (C);
  \end{tikzpicture}
\end{tikzpicture}
\caption{The moment polytope of $E$.  The projection of the Floer-essential Lagrangian $L$ is the purple dot, the projection of the families of Maslov-2 discs with boundary on $L$ are roughly the dashed black lines, and the contact hypersurface corresponding to the sphere bundle of radius $\frac{1}{\sqrt{\pi}}$ is the dashed purple line.}
\label{fig:polytope}
\end{figure}

\subsection{Setting up Floer theory}
\label{set-up-floer}
For simplicity, we set $\KK = \ZZ/2\ZZ$ in this section.  Thus, each orientation line $o_x$ appearing as a generator of a Floer complex is replaced by the corresponding critical point $x$.

$E$ is the complex line bundle 
\[
\begin{tikzcd}
\CC \arrow{r} & E \arrow{d}{\rho} \\
& \CC P^1
\end{tikzcd}
\]
with Chern class $c_1(E) = -[\omega_{FS}]$, where $\omega_{FS}$ is the Fubini-Studi form rescaled to give $\CC P^1$ an area of one.  The unit sphere bundle $SE$ has a contact form $\alpha$ satisfying $d\alpha = \rho^*\omega_{FS}$.  From the data $(SE, \alpha)$, $(\CC P^1, \omega_{FS})$, and the natural radial coordinate $r$ on the complex fibers, we equip $E\setminus\{r=0\}$ with the symplectic form $\Omega = \rho^*\omega_{FS} + d(\pi r^2\alpha)$ and extend to the zero section by $\Omega\big|_{\{r=0\}} = \rho^*\omega_{FS}\big|_{\CC P^1}$ (see \cite{geiges}, for example).

Away from the zero-section, $\Omega = d((1 + \pi r^2)\alpha)$.  The Reeb orbits of the contact form $(1+ \pi r^2)\alpha\big|_{r=1}$ traverse the fibers of the subbundle $SE\xlongrightarrow{\rho}\CC P^1$, and the simple orbits have period $1 + \pi$, \cite{ritter1}.

Fix a radius $R\in(0, \infty)$ and constant $C > 0$.  Let $\{H_n: E \longrightarrow\RR\}_{n\in\NN}$ be a family of Hamiltonians defined as follows.  We assume that each $H_n$ is everywhere of the form $h^n(\pi r^2)$ for a function $h^n:\RR\longrightarrow\RR$ that is 
\begin{enumerate}
\item convex and monotone increasing on $\RR_{\geq 0}$, 
\item bounded in absolute value by $C$ on $[0, \pi R^2]$, 
\item of slope $n(1 + \pi) + 1$ on $(\pi R_n^2, \infty)$, for some $R_n < R$,
\item and equal to $h^{n-1}$ on $[0, \pi R_{n-1}^2]$.  
\end{enumerate}
Further assume that the sequence $\{R_n\}$ tends to $R$ as $n$ tends to $\infty$ (see Figure \ref{fig:unbdd}).

\begin{figure}[htpb!]
\centering
\begin{tikzpicture}[scale=5]
\draw (0, -1.03660231) -- (0, -.07) node [left = 1pt] {$\RR$};
\draw (0, -1.03660231) -- (1, -1.03660231)node [right = 1pt] {$\pi r^2$};
\draw[violet] (0, -1.03660231) -- (.267948, -1);
\draw[->, violet]  (.267948, -1) -- (1, -.9)node [right=1pt] {$h^0$};
\draw[->, violet] (.267948, -1) -- (1, -.57735)node [right=1pt] {$h^1$};
\draw[->, violet]  (.641751, -.784185) -- (1, -.37797)node [right=1pt] {$h^2$};
\draw[->, violet] (.822168, -.5796118) -- (1, -.2582)node [right=1pt] {$h^3$};
\draw[->, violet] (.911356, -.418414) -- (1, -.1796)node [right=1pt] {.};
\draw[->, violet] (.955751, -.2988102) -- (1, -.12599) node [right=1pt] {.};
\draw[->, violet] (.977884, -.212366) -- (1, -.08873) node [right=1pt] {.};
\end{tikzpicture}
\caption{The family of Hamiltonians $\{h^n(\pi r^2)\}$}
\label{fig:unbdd}
\end{figure}
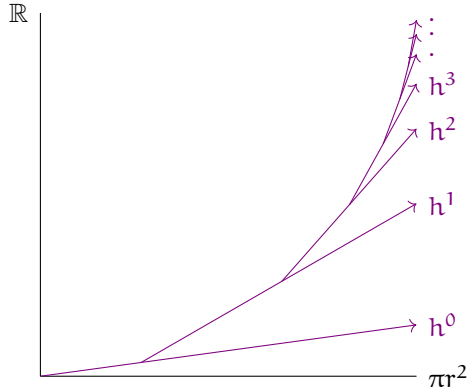

Choose $h^0$ so that $(h^0)'(0) = 0$.  This, together with conditions (1), (3), and (4) imply that the solutions of the differential equation
\begin{equation}
\dot{x}(t) = X_{H_n}(x(t))
\label{eq:pd}
\end{equation}
are the $\CC P^1$ family of constant orbits on the zero section and the $S^3$-worth of Reeb orbits for each period $k(1+\pi)$, $1\leq k\leq n$.  Condition (2) keeps the action of a critical point close to its symplectic area  and condition (4) also yields easier analysis of continuation maps.

To be able to set up Floer theory we need to either perturb the Hamiltonian or use Morse-Bott methods.  We apply a Morse perturbation directly to $\CC P^1$ to rid ourselves of the ``horizontally-degenerate'' critical points and then use Morse-Bott methods on the remaining $S^1$-degenerate families.

Let $F:\CC P^1\longrightarrow\RR$ be a $\cg{C}^2$-small Morse function with two critical points: a maximum value at $N$ (north) and a minimum value at $S$ (south).  The function $F$ has a Hamiltonian vector field $X_F$ defined through $\omega_{FS}$.  Consider solutions to the equations
\begin{equation}\left\{\begin{array}{c}
\frac{d}{dt}x(t) = X_{H_n}(x(t)) \\
\frac{d}{dt}(\rho\circ x)(t) = X_F(x(t))
\end{array}\right.\label{eq:pdpert}\end{equation}
Solutions of (\ref{eq:pdpert}) are precisely the solutions of (\ref{eq:pd}) with projection to $\CC P^1$ equal to $N$ or to $S$.

Each non-constant solution of (\ref{eq:pdpert}) occurs in an $S^1$-family.  Apply Morse-Bott methods to fix a maximum $M$ and a minimum $m$ of each $S^1$ family.  The soon-to-be-discussed differential will count solutions of a Floer equation using cascades.
Define $CF^*(H_n;\Gamma)$ to be the chain complex over $\Gamma$ generated by the maximum $M$ and minimum $m$ of each family of solutions of (\ref{eq:pdpert}). 

The generators of $CF^*(H_n;\Gamma)$ occur in six flavors: there are 4 Reeb orbits of period $k(1+\pi)$ for each $k\in\{1, ..., n\}$ that correspond to a unique choice of $S$ or $N$ and $m$ or $M$.  There are also the minimum and maximum constant orbits $N$ and $S$ themselves.  Let us fix notation.  $x^0_-$ is the constant orbit mapping to $N$ and $y^0_-$ is the constant orbit mapping to $S$.  $x^k_+$ is the maximum $M$ of the family of period-$k(1+\pi)$ orbits lying above $N$ and $x^k_-$ is the minimum $m$.  $y^k_+$ is the maximum $M$ of the family of period-$k(1+\pi)$ orbits lying above $S$ and $y^k_-$ is the minimum $m$.  

An element
\[
T^sz; z\in\{x^k_{\pm}, y^k_{\pm}\}, s\in\RR
\]
of $CF^*(H_n; \Gamma)$ is $\RR$-graded by
\begin{equation}
\label{eq:cz}
\mu(T^sz) = -2k + 2s + \mu_F(\rho\circ z) \pm \frac{1}{2},
\end{equation}
where $\mu_F(\rho\circ z)$ is the Morse index of the critical point of $F$ corresponding to $\rho\circ z.$  Note that these are cohomological gradings.

Lift a period-$k$ Reeb orbit $z$ to the $k$-fold fiber disc $\tilde{z}$.  Choose a generic, cylindrical almost-complex structure $J$.  The differential $\dd^{fl}$ counts rigid solutions of
\begin{equation}
\frac{\dd u}{\dd s} + J\left(\frac{\dd u}{\dd t} - X_{H_n} - \rho^*X_F\right) = 0,
\label{eq:floer2}
\end{equation}
so that if $u$ is a rigid solution of (\ref{eq:floer2}) with positive limit $y$ and negative limit $x$, $u$ contributes a term $T^{\omega(-\tilde{y}\#w\#\tilde{x})}x$ to $\dd^{fl}(y)$.

The one-form $\alpha$ on $E$ induces a splitting $TE\cong V\oplus H$ into vertical and horizontal components, where, for $e\in E$, $V_e\cong\CC$ and $H_e\cong T_{\rho(e)}\CC P^1$.  Let $(i_t)_{t\in S^1}$ be an $S^1$-family of almost-complex structures on $\CC$, each compatible with the standard symplectic structure, and each agreeing with the standard complex structure outside of a small neighborhood of zero.  Let $(j_t)_{t\in S^1}$ be an $S^1$-family of almost-complex structures on $\CC P^1$, each compatible with $\omega_{FS}$.  Finally, denoting the space of linear maps $H\longrightarrow V$ by $L(H, V)$, let $(B_t)_{t\in S^1}\in\Gamma(E\times S^1, L(H, V))$ satisfy $i_tB_t + B_tj_t = 0$ for each $t\in S^1$.  Further assume that the support of each $B_t$ lies close to the zero section and outside of a neighborhood of the fibers above $N$ and $S$. 

We restrict to the space $\cg{J}$ of $\Omega$-tame almost-complex structures of the form
\[
J = (J_t)_{t\in S^1} = \left(\left[\begin{array}{cc}
i_t & B_t \\
0 & j_t\end{array}\right]\right)_{t\in S^1}
\]
with respect to the splitting $TE\cong V\oplus H$.  We call $J$ {\it standard} at $e\in E$ if, for each $J_t$,
\[
(J_t)_e = \left[\begin{array}{cc} i & 0 \\ 0 & j_t\end{array}\right].
\]

\begin{lemma}[Albers-Kang, \cite{albers-k}]
The set $\cg{J}_{reg}$ of almost-complex structures $J\in\cg{J}$ such that
\begin{enumerate}
\item finite energy solutions of (\ref{eq:floer2}) are regular,
\item finite energy solutions $u:\RR\times S^1\longrightarrow\CC P^1$ of $\dd_s u + j_t(\dd_t u - X_F) = 0$ are regular, and
\item $S^1$-families of simple $J_t$-holomorphic spheres are regular
\end{enumerate}
is of Baire second category.
\end{lemma}
\begin{remark}
Fix $J = (J_t)_{t\in S^1}\in\cg{J}_{reg}$.  We must check that $J_t$-sphere bubbles do not contribute to limit points of the relevant moduli spaces.  We assume that each $J_t$ is standard on the annulus bundle containing all non-constant periodic orbits of the family $\{H_n\}_{n\in\NN}$.  The maximum principle ensures that the non-constant periodic orbits do not intersect the moduli space of $J_t$-holomorphic spheres, for any $t\in S^1$.  As $E$ is monotone, standard energy and index arguments show that the bubbling off of $J_t$-holomorphic spheres of Chern number not equal to one does not occur.  So consider the moduli space $\cg{M}(J; 1)$ of $S^1$-families of $J_t$-holomorphic spheres of Chern number equal to one.  The elements of $\cg{M}(J; 1)$ form a codimension-one subset of $E$.  If an element of $\cg{M}(J; 1)$ appears as a limit point of a sequence of Floer solutions of virtual dimension 1, then this gives rise to a Floer solution of virtual dimension $-1$, which is therefore a one-periodic orbit of some Hamiltonian $H$.  The sphere bubble must intersect this one-periodic orbit, and so this one-periodic orbit is constant.  But the constant one-periodic orbits form a dimension-zero subset of $E$, and so, after a small, generic perturbation of $H$, do not intersect $\cg{M}(J; 1)$.  (See Chapter 3 in \cite{salamon} for a thorough discussion of bubbling in monotone manifolds.)
\end{remark}

\subsection{Computing the differential}
We use Lemmas \ref{lemma:a-k} -- \ref{lemma:ritter} to determine the differential of $\widehat{SC^*}(E; \Lambda)$.  A cartoon of the Floer complex is given in Figure \ref{fig:fullcomplex}.  The ``horizontal'' differentials correspond to Floer trajectories in the fiber above a critical point.  The ``diagonal'' differentials correspond to Floer trajectories whose projection onto $\CC P^1$ either covers all of $\CC P^1\setminus\{N, S\}$ (in the case of an arrow from $x$ to $y$) or is a Morse flow-line of $F$ (in the case of an arrow from $y$ to $x$).

\begin{figure}[htbp!]
\[
\begin{tikzcd}[column sep=1cm]
x^0_- & x^1_+ \arrow{l} \arrow{dl} & Tx^1_- & \arrow{l} \arrow{dl} \dots & \dots & T^{n-1}x_+^n \arrow{dl} \arrow{l} & T^nx_-^n \\
Ty^0_- & Ty^1_+ \arrow{l} \arrow{ur} & T^2y^1_- & \arrow{l} \dots & \dots & T^ny_+^n \arrow{l} \arrow{ur} & T^{n+1}y_-^n
\end{tikzcd}
\]
\caption{The trajectories contributing to $CF^*(H_n; \Gamma)$}
\label{fig:fullcomplex}
\end{figure}
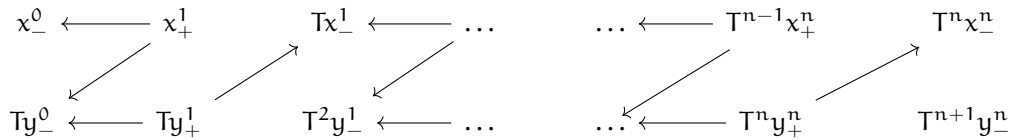

\begin{lemma}[Albers-Kang, \cite{albers-k}]
Any trajectory of $X_{H_n}+\rho^*X_F$ with vanishing symplectic area and with both asymptotic limits contained in the same fiber remains wholly in that fiber.  Thus, any such trajectory is identified with a trajectory of $X_{h^n(\pi r^2)}$ in $\CC$.
\label{lemma:a-k}
\end{lemma}
Utilizing grading considerations and the fact that $SH^*(\CC; \Lambda) = 0$ (see \cite{seidel}), the differential restricted to generators in the complex lines above $N$ and $S$ are the horizontal differentials shown in Figure \ref{fig:fullcomplex}.

\begin{lemma}
The winding number of the one-periodic orbits of $X_{H_n} +\rho^*X_F$ is decreased by the differential.  The only Floer trajectories with both asymptotes contained in the same Reeb orbit have image identically equal to this Reeb orbit.
\label{lemma:winding1}
\end{lemma}
\begin{proof}
We adapt Lemma 2.3 in \cite{cieliebak-o} to this setting and use the notation of Section 5.1 in \cite{frauenfelder}.  For $p\in\CC P^1$ let $E_p = \rho^{-1}(p)$.  For a path $\gamma(t)\in\CC P^1$, $T'\in\RR$, and $v\in E_{\gamma(T')}$, let $P^T_{\gamma(t)}v$ be the parallel transport of $v$ along $\gamma$ to $E_{\gamma(T)}$ with respect to the fixed connection given by $\alpha$.

Let $u(s, t)$ be a Floer trajectory with negative asymptote a non-constant orbit $x(t)$.  Let $v(s, t) = \rho\circ u(s, t)$.  For fixed $\mathfrak{s}\in\RR$ denote by $v^{\mathfrak{s}}(t)$ the path $v(\mathfrak{s}, t)$, and for fixed $\mathfrak{t}\in S^1$ denote by $v^{\mathfrak{t}}(s)$ the path $v(s, \mathfrak{t})$.  Let $\dd_su^v$, respectively $\dd_tu^v$, be the vertical component of $\dd_su$, respectively $\dd_tu$, under the splitting $TE\cong V\oplus H$ determined by $\alpha$.

Let $\sigma'$ be the largest value in $\RR\cup\{\infty\}$ so that $J_{u(s, t)}$ is standard for all $s < \sigma'$ and all $t\in S^1$.  Note that if $u(s, t)$ stays at constant radius then $\sigma' = \infty$.  Choose $\sigma < \sigma'$ so that $u(\sigma, t)$ lies in a neighborhood of $u(\sigma', t)$ for each $t\in S^1$.  Define a map $u_0:\RR\times S^1\longrightarrow E_{u(\sigma, 0)} \cong\CC$ by $u_0(s, t) = P_{v^0(s)}^{\sigma}P_{v^s(t)}^0u(s, t)$.  If $F_{\alpha}$ is the curvature of $(E, \alpha)$ and $R_{\alpha}$ is the Reeb vector field of $\alpha$, then
\[
\dd_su_0(s, t) = -\int_0^t F_{\alpha}(\dd_sv, \dd_tv)dt\cdot R_{\alpha}(u_0) + P_{v^0(s)}^{\sigma}P_{v^s(t)}^0\dd_su^v(s, t),
\]
and
\[
\dd_tu_0 = P_{v^0(s)}^{\sigma}P_{v^s(t)}^0\dd_tu^v(s, t).
\]
As $u$ is a Floer trajectory and $R_{\alpha}$ is invariant under parallel transport, we deduce that $u_0\big|_{\RR_{\leq\sigma}\times S^1}$ satisfies a Floer equation 
\begin{equation}
\label{eq:floer-fiber}
\dd_su_0 + \int_0^t F_{\alpha}(\dd_sv, \dd_tv)dt\cdot R_{\alpha} + i\left(\dd_tu_0 - X_{H_n}\right).
\end{equation}
Write $u_0(s, t) = (a(s, t), f(s, t))$ in the coordinates $\RR\times S^1$ on $\CC^*$ induced by the standard Hermitian metric.  Integrating over the radial direction of (\ref{eq:floer-fiber}),
\begin{equation}
\int_{S^1}\dd_s a(s, t)dt = \int_{S^1}\alpha(\dd_tf(s, t))dt - \int_{S^1}(h^n)'(\pi a(s, t)^2)dt
\end{equation}
for fixed $s\in\RR_{\leq\sigma}$.  As $\int_{S^1}\alpha(\dd_tf(s, t))$ is constant and $h^n$ is convex, we deduce that either
\[
\int_{S^1}\dd_s a(s, t)dt > 0 \text{ for some } s \in(-\infty, \sigma)
\]
or 
\[
\int_{S^1}\dd_s a(s, t)dt = 0 \text{ for all } s\in(-\infty, \sigma).
\] 
In the former case, there exists $(\mathfrak{s}, \mathfrak{t})\in\RR\times S^1$ such that $a(\mathfrak{s}, \mathfrak{t}) > \lim\limits_{s\rightarrow-\infty}a(s, t)$.  As parallel transport preserves radius, $u(\mathfrak{s}, \mathfrak{t})$ leaves the disc bundle containing $x(\cdot)$.  The maximum principal implies that $\lim\limits_{s\rightarrow\infty}u(s, \cdot)$ lives at a larger radius than $x(\cdot)$.  The convexity of $h^n$ now proves the lemma.

In the latter case, $u_0(s, t)$ remains in the sphere bundle containing $x(t)$ for all $(s, t)\in(-\infty, \sigma)$.  Letting $\sigma\rightarrow\sigma'$, we can either argue as above, or $\sigma' = \infty$ and $u_0(s, t)$ stays at constant radius for all $s\in(-\infty, \infty)$.  This implies that $u(s, t)$ remains in a sphere bundle of constant radius.

If $\lim\limits_{s\rightarrow\pm\infty}u(s,\cdot) = x$, then the image of $u(s, t)$ is contained in a sphere bundle $S_R$ of some radius $R$.  By the exactness of $\Omega$ on $S_R$, $u$ must have energy $E(u) = 0$, and so $u(s, t)$ is constant.

\end{proof}

\begin{lemma}[Albers-Kang, \cite{albers-k}]
If $u$ is a Floer solution of $X_{H_n} + \rho^*X_F$ then $\rho\circ u$ is a Floer solution of $X_F$; in particular, the Conley-Zhender index of critical points of $X_F$ increases from the positive to the negative asymptote of the trajectory $\rho\circ u$.
\label{lemma:arnold}
\end{lemma}

\begin{lemma}[Ritter, \cite{ritter1}]
\label{lemma:ritter}
$HF^*(H_n; \Lambda)\cong\bigslant{\Lambda[x]}{(x^2 + T)}.$  The induced continuation map on cohomology is 
\begin{align*}
HF^*(H_n; \Lambda)&\longrightarrow HF^*(H_{n+1}; \Lambda) \\
1 &\mapsto 1 \\
x &\mapsto T.
\end{align*}
\end{lemma}
Lemma \ref{lemma:winding1} also holds for continuation maps induced by $\RR$-families of Hamiltonians that are monotone-decreasing in $\RR$.  We may thus choose continuation maps that act as the canonical inclusions.  Allowing all trajectories of index one that satisfy Lemmas \ref{lemma:a-k}, \ref{lemma:winding1}, and  \ref{lemma:arnold}, that define a differential that squares to zero, and that yield Lemma \ref{lemma:ritter}, produces the complex shown in Figure \ref{fig:fullcomplex}.

Let $E_R\subset E$ be the disc bundle of radius $R$.  The completed symplectic cochain complex of $E_R$ in degree $k$ is
\[
\widehat{SC^k}(E_R; \Lambda) =\Lambda\left\la\left\{\sum_{i=0}^{\infty}a_iT^{s_i}z_i\hspace{.1cm}\bigg|\hspace{.1cm} a_i\in\KK; \exists\text{ $n$ s.t. }T^{s_i}z_i\in CF^{k}(H_n; \Gamma); \lim_{i\rightarrow\infty}\cg{A}(T^{s_i}z_i)=\infty\right\}\right\ra.
\]
Following Albers-Kang, we can rephrase the action in terms of symplectic area \cite{albers-k}.  The first observation is that, for a critical point $T^sz$, where $z = x^n_{\pm}$ or $y^n_{\pm}$, the index formula can be manipulated:
\[
k = -2n + 2s\omega(A) + \mu_F(\rho(z)) \pm \frac{1}{2}\iff
n = -\frac{1}{2}\left(k - 2s - \mu_F(\rho(z)) \mp \frac{1}{2}\right).
\]
Thus, the action of $T^sz$ can be reformulated as
\begin{align*}
\cg{A}_{H_n}(T^sz) &= s -\int_D \tilde{z}^*\Omega + \int_0^1 H_n(z(t)) + \rho^*F(z(t)) dt 
\\ &\cong s -\int_D (d^n)^*d(\pi R_n^2\alpha) + \int_0^1 H_n(z(t)) + \rho^* F(z(t)) dt\\
&= s -n\pi R_n^2 + \int_0^1 H_n(z(t)) + \rho^*F(z(t))dt \\
&= s + \frac{\pi R_n^2}{2}\left(k - 2s - \mu_F(\rho(z)) \mp\frac{1}{2}\right) + \int_0^1 H_n(z(t)) + \rho^* F(z(t))dt\\
&= (1 - \pi R_n^2)s + C(z),
\end{align*}
where $C(z) = \frac{\pi R_n^2}{2}\left(k - \mu_F(\rho(z)) \mp \frac{1}{2} \right)+ \int_0^1 H_n(z(t)) + \rho^* F(z(t))dt$ is uniformly bounded.

\begin{lemma}
The completed symplectic cohomology of a disc bundle of radius $R$ is 
\[
\widehat{SH^*}(E_R; \Lambda) \cong \left\{\begin{array}{cc} 0 & R < \frac{1}{\sqrt{\pi}} \\ \Lambda & R \geq \frac{1}{\sqrt{\pi}}\end{array}\right..
\]
\end{lemma}
\begin{proof}
\begin{enumerate}
\item Suppose $\pi R^2 < 1$.  Then $\lim\limits_{i\rightarrow\infty}\cg{A}(T^{s_i}z_i) = \lim\limits_{i\rightarrow\infty} s_i$.
Thus, $\widehat{SC^{-\frac{1}{2}}}(E_R; \Lambda)$ includes the element
\begin{equation}
Z = \sum_{i=1}^{\infty} T^{i-1}x_+^{i} + T^{i}y_+^i.
\label{ob:cobound}
\end{equation}
The differential applied to $Z$ yields
\begin{align*}
\dd(Z) &= \sum_{i=1}^{\infty} T^{i-1}x_-^{i-1} + T^{i}y_-^{i-1} + T^{i}x_-^i + T^{i}y_-^{i-1} \\
&= x_-^0.
\end{align*}
By $T$-linearity of the differential, this computation extends to produce an annihilator for any element of the form $T^kx_-^0$.

$T^kx_-^0$, $T^{k+i}x_-^i$, and $T^{k+i+1}y_-^i$ are equivalent in cohomology, and so every cocycle generating $\widehat{SC^*}(E; \Lambda)$ is killed by a completed coboundary.  Similarly, any completed cocycle is killed by formally adding together the annihilators of the individual summands (by construction this formal sum will be an element of $\widehat{SC^*}(E_R; \Lambda)$).  Thus, $\widehat{SH^*}(E_R; \Lambda) = 0.$

\item If $\pi R^2 > 1$ the infinite sum (\ref{ob:cobound}) is no longer an element of $\widehat{SC^*}(E_R; \Lambda).$  The cohomology theory reduces to the uncompleted version and is therefore of rank one.

\item Finally, suppose $\pi R^2 = 1$.  By the assumptions of boundedness and convexity on each $h^n$, as well as the assumption that $(h^n)'(\pi R_n^2) = n(1+\pi) + 1$, it follows that 
\[
\sum\limits_{n=0}^{\infty}\left(n(1+\pi) + 1\right)\left(\pi R_{n+1}^2 - \pi R_n^2\right) < \infty.
\]
Therefore, $O(\pi R_{n+1}^2 - \pi R_n^2) < \frac{1}{n^2}$ as $n\rightarrow\infty$.  Because $\lim\limits_{n\rightarrow\infty} R_n = \frac{1}{\sqrt{\pi}}$, this implies that $1 - \pi R_n^2 < \frac{1}{n}$ for large enough $n$, and so $0 < (1 - \pi R_n^2)\left(n(1+\pi) + 1\right) < 2 + \pi$.  We deduce that the limit of the action of generators comprising the sum $Z$ in equation (\ref{ob:cobound}) is finite.  As in the case $\pi R^2 > 1$, we conclude that $\widehat{SH^*}(E_R; \Lambda) \cong \Lambda$.
\end{enumerate}
\end{proof}

A similar computation shows that 
\begin{lemma}
The completed symplectic homology of the disc bundle of radius $R$ is
\[
\widehat{SH_*}(E_R; \Lambda) \cong \left\{\begin{array}{cc} 0 & R \leq \frac{1}{\sqrt{\pi}} \\ \Lambda & R > \frac{1}{\sqrt{\pi}}\end{array}\right..
\]
\end{lemma}
Theorem \ref{prop:example} when $R_1\leq \frac{1}{\sqrt{\pi}}$ now follows from the long-exact sequence (\ref{diag:les-cob}).  The case $R_1 > \frac{1}{\sqrt{\pi}}$ follows from the following lemma.

\begin{lemma}
The map $\mathfrak{c}:\widehat{SH_*}(E_{R_1}; \Lambda)\longrightarrow\widehat{SH^*}(E_{R_2}; \Lambda)$ defined in (\ref{eq:continuation}) is an isomorphism whenever $R_1 > \frac{1}{\sqrt{\pi}}$.
\end{lemma}

\begin{proof}
By Poincar\'e duality, $CF_*(H_0; \Gamma)$ is isomorphic (up to grading) to the cochain complex defined by the Hamiltonian $-H_0 - F\circ\rho$, which we denote by $CF^*(-H_0; \Gamma)$.  We will abuse notation and continue to denote the generators of $CF^*(-H_0; \Gamma)$ by $x_-^0$ and $y_-^0$.  Let $c^{-1}:CF^*(-H_0; \Gamma)\longrightarrow CF^*(H_0; \Gamma)$ be a continuation map.  If $R_1 > \frac{1}{\sqrt{\pi}}$ then $\widehat{SC_*}(E_{R_1}; \Lambda)$ and $\widehat{SC^*}(E_{R_2}; \Lambda)$ are canonically isomorphic to the uncompleted theories, and the map $\widehat{SC_*}(E_{R_1}; \Lambda)\xlongrightarrow{\mathfrak{c}}\widehat{SC^*}(E_{R_2}; \Lambda)$ is determined by the image of $x_-^0 + T^{-1}y_-^0\in CF_*(H_0; \Gamma)$ under the composition 
\[
CF_{-*}(H_0; \Gamma)\cong CF^*(-H_0; \Gamma) \xlongrightarrow{c^{-1}} CF^*(H_0; \Gamma)\xhookrightarrow{} SC^*(E; \Gamma).
\]

Recall the maps $\iota^{-1}_{(a, b)}$ and $\iota^0_{(a, b)}$ from Section \ref{non-vanish}, define through a Morse-Smale pair $(f, g_L)$ on $L$, and suppose that $f$ has a unique minimum $p$.  Analogous maps $\iota^{-1}$, respectively $\iota^0$, are defined from the (untruncated) Floer complexes $CF^*(-H_0; \Gamma)$, respectively $CF^*(H_0; \Gamma)$ to $CF^*(L; \Gamma)$.  The proof of Lemma \ref{lem:descent} extends to the equality 
\begin{equation}
(\iota^0\circ c^{-1})^* = (\iota^{-1})^*.
\label{eq:commute}
\end{equation}
We will use this identity to understand the map $c^{-1}$.

Let $\cg{M}^2(L, p)$ be the Maslov index-2 discs $u:(D^2, \dd D^2)\longrightarrow (M, L)$ with $u(1) = p$.  We have
\[
\sum_{u\in\cg{M}^2(L, p)}[\dd u] = 0
\]
(see \cite{auroux}, \cite{ritter-s}).
An index calculation now shows that the quantum differential $\dd$ on $HF^*(L; \Lambda)$ is the ordinary differential on $H^*(L; \Lambda)$ (Proposition 6.1.4 (a) in \cite{biran-c2}).  Therefore, $p$ is the only representative of the unit of  $HF^*(L; \Gamma)$ in $CF^*(L; \Gamma)$.  To analyze the contributions to the unit of $\iota^0$ and $\iota^{-1}$, it therefore suffices to analyze the pearly/Floer trajectory amalgamates that negatively asymptote to $p$.

Let $g_{\CC}$ be the standard metric on $\CC$ and let $g_{\CC P^1}$ be the standard metric on $\CC P^1$.  Let
\[
g = \left[\begin{array}{cc} g_{\CC} & 0 \\ 0 & g_{\CC P^1}\end{array}\right]
\]
be a metric on $E$ with respect to the splitting $TE\cong V\oplus H$, as in Subsection \ref{set-up-floer}.  Choose a generic almost-complex structure $J$.  Denote the quantum cochain complex associated to a Morse-Smale pair $(F, g)$ on $M$ by $QC^*(F)$.
Consider a map $\phi^{-1}: QC^*(-H - F\circ\rho)\longrightarrow QC^*(L)$, respectively $\phi^0: QC^*(H + F\circ\rho)\longrightarrow QC^*(L)$, that counts rigid configurations of the type shown in Figure \ref{fig:pearly-mod1}.  Explicitly, $\phi^{-1}(x)$, respectively $\phi^0(x)$, is the count of rigid configurations $(u_1, ..., u_{\ell})$ such that
\begin{enumerate}
\item $u_i: (D^2, \dd D) \longrightarrow (M, L)$ is a $J$-holomorphic disc that is non-constant if $i < \ell$,
\item $u_1(1) = p$,
\item there exists $t\in(-\infty, 0)$ such that $\Phi_t(u_{i+1}(1)) = u_i(-1)$ for all $i < \ell$, where $\Phi_t$ is the time-$t$ flow of $f$, and
\item there exists a flow line $\beta(t)$ of $-H^{\tau_{0}} - \rho^* F$, respectively $H^{\tau_0} + \rho^*F$, with $\lim\limits_{t\rightarrow\infty}\beta(t) = x$ and $\beta(0) = u_{\ell}(0, 0)$.
\end{enumerate}
As in Subsection \ref{lqc}, we only consider such configurations up to action by $Aut(D^2, \pm 1)^{\ell -1}$, where $Aut(D^2, \pm 1)$ is the set of automorphisms of $D^2$ fixing $\pm 1$.  The maps $\phi^{-1}$ and $\phi^0$ are the unital component of the dual of the quantum inclusion map studied in Section 5.4 of \cite{biran-c2}.

If $\pi_p:CF^*(L; \Lambda)\rightarrow\Lambda\cdot p$ is the projection onto the $\Lambda$-span of $p$, then under the PSS isomorphism,
\begin{equation}
(\pi_p\circ\iota^{-1})^* = (\phi^{-1})^*\hspace{1cm}\text{and}\hspace{1cm}(\pi_p\circ\iota^0)^* = (\phi^0)^*.
\end{equation}
\begin{figure}[htbp!]
\centering
\includegraphics[scale=.7]{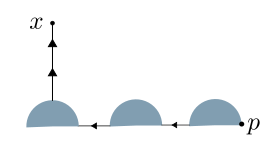}
\caption{Configurations defining $\phi^{-1}(x)$ and $\phi^0(x)$},
\label{fig:pearly-mod1}
\end{figure}

The dimension-zero configurations $(u_1, ..., u_{\ell})$ have $|x| = \mu(u_1) + ... + \mu(u_{\ell})$, where $|x|$ is the Morse grading and $\mu(u_i)$ is the Maslov index of $u_i$, \cite{biran-c2}.  Thus, $\phi^0(y_-^0)$ is a multiple of $T^0 = 1$ and $\phi^0(x_-^0)$ is a multiple of $T$.  In fact, there is precisely one gradient trajectory $\beta(t)$ with $\lim\limits_{t\rightarrow\infty}\beta(t) = y$ and $\beta(0) = p$, and so $\phi^0(y_-^0) = 1$.  This is the yellow curve in Figure \ref{fig:spiked1}.  

The configurations contributing to $\phi^0(x_-^0)$ look like a single Maslov index-2 disc $u$ with $u(1) = p$ and $u(0, 0)$ intersecting a gradient flow line that converges at positive infinity to $x_-^0$.  There is one such configuration, represented by the green curve in Figure \ref{fig:spiked1}.  Thus, $\phi^0(x_-^0) = T$.

Similarly, $\phi^{-1}(y_-^0) = 0$ and $\phi^{-1}(x_-^0) = T$, where the only contributing configuration is represented by the blue curve in Figure \ref{fig:spiked2}.

Using equation (\ref{eq:commute}) and the fact that $c^{-1}$ preserves the Conley-Zehnder index (\ref{eq:cz}), we deduce that 
\[
(c^{-1})^*(y_-^0) \in\left\{ 0, x_-^0 + Ty_-^0 \right\}\hspace{1cm}\text{and}\hspace{1cm}(c^{-1})^*(x_-^0) \in \left\{ T^{-1}x_-^0, y_-^0\right\}.
\]
As $[x_-^0 + Ty_-^0] = 0$ in $SH^*(E; \Lambda)$, $(c^{-1})^*(y_-^0) = 0$.  And $[T^{-1}x_-^0] = [y_-^0]$ generates $SH^*(E; \Lambda)$, so $(c^{-1})^*(x_-^0)$ generates $SH^*(E; \Lambda)$.  Thus, $(c^{-1})^*(x_-^0 + T^{-1}y_-^0)$ generates $SH^*(E; \Lambda)$.  We conclude that the map $\mathfrak{c}:SH_*(E_{R_1}; \Lambda)\longrightarrow SH^*(E_{R_2}; \Lambda)$ is an isomorphism.

\begin{figure}[htbp!]
\begin{subfigure}{.5\textwidth}
\centering
\begin{tikzpicture}
\centering
    \begin{axis}[
		hide axis,
		axis equal,
		domain=0:2.5,
		view={0}{90},
	]

    \draw [color=white, fill=teal!4] (10,0) -- (25,0) -- (25,25) -- (0,25) -- (0,10) -- cycle;

    \coordinate  (A) at (10,0);
    \coordinate (B) at (25,0);
    \coordinate (C) at (0,25);
    \coordinate (D) at (0,10);
    
    \coordinate (E) at (10, 10);
    \coordinate (F) at (5, 5);
    \draw[yellow] (E) -- (F);
    \draw[dotted] (E) -- (F);
    \draw[green] (13, 0) arc (-20:55:9);
    \draw[dotted] (13, 0) arc (-20:55:9);
    \draw[dotted] (0, 13) arc (110:35:9);
    \draw (10, 10) node[circle, fill=violet, scale=.5] [label=right:$L$] {};
    
    \draw[dashed, violet] (0, 20) -- (20, 0);

    \draw [<->, line width=.5pt] (B) -- (A) -- (D) -- (C);
    
    \addplot3[gray,
			quiver={
			 u={-2*x},
			 v={-2*y},
			 scale arrows=0.04,
			},
			-stealth,samples=10]
				{-x^2 - y^2};
				
   \draw [color=white, fill=white!4] (11,-1) -- (0,0) -- (-1,11) -- cycle;
   \draw[line width=.5pt] (10, 0) -- (0, 10);
   \draw[line width=.5pt, yellow!] (0, 10) -- (5, 5);
   \draw[ -stealth, yellow!] (2, 8) -- (1, 9);
   \draw[ -stealth, yellow!] (4, 6) -- (3, 7);
   \draw[ -stealth, gray] (7, 3) -- (6, 4);
   \draw[ -stealth, gray] (9, 1) -- (8, 2);
   \draw[line width = .5 pt, green] (10, 0) -- (13, 0);
   
   \draw (10, 0) node[circle, fill, scale =.25] [label=above:$x$] {$y$};
   \draw (0, 10) node[circle, fill, scale=.25] [label=right:$y$] {$x$};
    \end{axis}
\end{tikzpicture}
\caption{Spiked trajectories of $H^{\tau_0} + F\circ\rho$}
\label{fig:spiked1}
\end{subfigure}
\begin{subfigure}{.5\textwidth}
\centering
\begin{tikzpicture}
\centering
    \begin{axis}[
		hide axis,
		axis equal,
		domain=0:2.5,
		view={0}{90},
	]

    \draw [color=white, fill=teal!4] (10,0) -- (27,0) -- (27,30) -- (0,30) -- (0,10) -- cycle;

    \coordinate  (A) at (10,0);
    \coordinate (B) at (27,0);
    \coordinate (C) at (0,30);
    \coordinate (D) at (0,10);
    
    \coordinate (E) at (10, 10);
    \coordinate (F) at (5, 5);
    \draw[blue] (E) -- (F);
    \draw[dotted] (E) -- (F);
    \draw[dotted] (13, 0) arc (-20:55:9);
    \draw[dotted] (0, 13) arc (110:35:9);
    \draw (10, 10) node[circle, fill=violet, scale=.5] [label=right:$L$] {};
    
    \draw[dashed, violet] (0, 20) -- (20, 0);

    \draw [<->, line width=.5pt] (B) -- (A) -- (D) -- (C);
    
    \addplot3[gray,
			quiver={
			 u={2*x},
			 v={2*y},
			 scale arrows=0.04,
			},
			-stealth,samples=10]
				{-x^2 - y^2};
				
   \draw [color=white, fill=white!4] (11,-1) -- (0,0) -- (-1,11) -- cycle;
   \draw[line width=.5pt] (10, 0) -- (0, 10);
   \draw[ -stealth, blue!] (8,2) -- (9,1);
   \draw[ -stealth, gray!] (3,7) -- (4,6);
   \draw[ -stealth, blue] (6,4) -- (7,3);
   \draw[ -stealth, gray] (1,9) -- (2,8);
   \draw[line width = .5 pt, blue] (10, 0) -- (5, 5);
   
   \draw (10, 0) node[circle, fill, scale =.25, left] [label=above:$x$] {$y$};
   \draw (0, 10) node[circle, fill, scale=.25] [label=right:$y$] {$x$};
    \end{axis}
\end{tikzpicture}
\caption{Spiked trajectories of $-H^{\tau_0} - F\circ\rho$}
\label{fig:spiked2}
\end{subfigure}
\caption{The configurations contributing to $\phi^0$, respectively $\phi^{-1}$, depicted in Figure (a), respectively Figure (b).  The vector fields are approximations of the gradients of $\pm (H^{\tau_0} + F\circ\rho)$; they are, in fact, the gradients corresponding to the Morse-Bott pairs $\pm(H^{\tau_0}, F)$.}
\label{fig:spiked}
\end{figure}
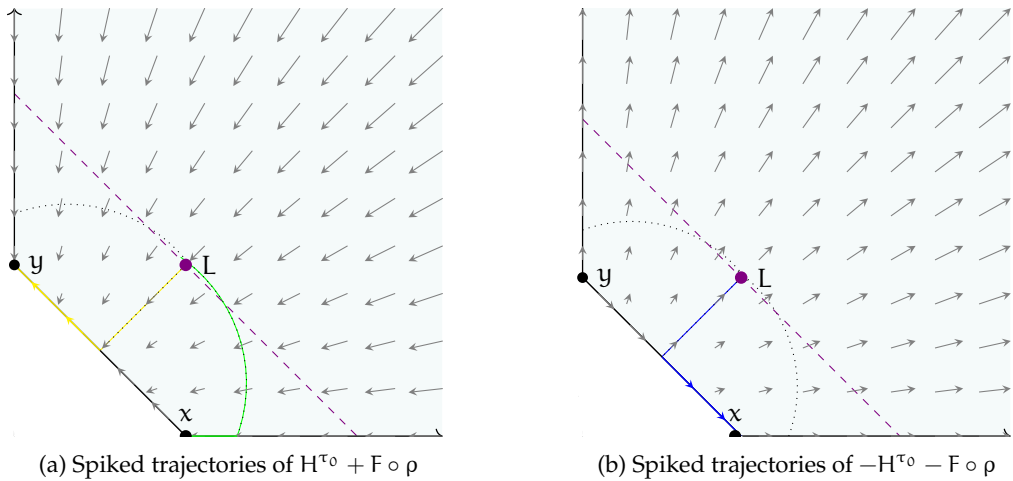

\end{proof}  

By the long exact sequence (\ref{diag:les-cob}), $\widehat{SH^*}(W) = 0$ when $R_1 > \frac{1}{\sqrt{\pi}}$.  This concludes the proof of Theorem \ref{prop:example}.

\section{Closed mirror symmetry predictions}
\label{hms}
As a generalization of the setup in Section \ref{example}, let $E_{m, k}$ be the complex line bundle $\cg{O}(-k) \xlongrightarrow{\rho} \CC P ^m$.  Equip $E_{m, k}\setminus\{r=0\}$ with the symplectic form $\Omega = \rho^*\omega_{FS} + d(k\pi r^2\alpha)$, where, as in Section \ref{example}, $\omega_{FS}$ is the rescaled Fubini-Studi form giving $\CC P^m$ an area of one, and $\alpha$ is a contact form on $SE$ with $d\alpha = \rho^*\omega_{FS}$.  Extend over the zero section by $\Omega\big|_{\{r=0\}} = \rho^*\omega_{FS}$.  We restrict to $1 \leq k \leq m$, in which case $E_{m, k}$ is monotone.

$E_{m, k}$ is a toric variety whose image under the moment map is
\[
\Delta := \left\{(v_1, ..., v_{m+1})\in\RR^{m+1}\hspace{.1cm}\bigg|\hspace{.1cm}v_i\geq 0\hspace{.05cm}\forall\hspace{.05cm} i\in\{1, ..., m+1\}; \hspace{.1cm} -v_1 - ... - v_m + kv_{m+1} \geq -1\right\}
\]
(see, for example, Subsection 7.6 in \cite{ritter} or Subsection 12.5 in \cite{ritter-s}).

Let $\KK = \CC$ and set $\Lambda^* = \Lambda\setminus\{0\}$.  Recall the valuation $val$, defined in (\ref{val}).  The mirror of $E_{m, k}$ is the subset of $(\Lambda^*)^{m+1}$ given by
\[
E^{\vee}_{m, k} := \left\{(z_1, ..., z_{m+1})\in(\Lambda^*)^{m+1}\hspace{.1cm}\big|\hspace{.1cm}\left(val(z_1), ..., val(z_{m+1})\right)\in\Delta^{\mathrm{o}}\right\},
\]
equipped with superpotential
\begin{align}
\cg{W}:E^{\vee}_{m, k}&\longrightarrow\Lambda \\
(z_1, z_2, ..., z_{m+1}) &\mapsto z_1 + z_2 + ... + z_m + z_{m+1} + Tz_1^{-1}z_2^{-1}...z_m^{-1}z_{m+1}^{k}.
\end{align}
(See Example 7.12 in \cite{ritter} or Proposition 4.2 in \cite{auroux}.)  Mirror symmetry predicts an isomorphism between the symplectic cohomology of a toric variety and the Jacobian of $\cg{W}$.  For example, computations in \cite{ritter-s} confirm that 
\begin{equation}
SH^*(E_{m, k}; \Lambda)\cong\bigslant{\Lambda[z_1^{\pm}, z_2^{\pm}, ..., z_{m+1}^{\pm}]}{(\dd_{z_1}W, \dd_{z_2}W, ..., \dd_{z_{m+1}}W)} =: Jac(\cg{W}).
\label{eq:jac}
\end{equation}

This story generalizes to domains of restricted size.  Let $D_RE_{m, k}$ be the disc bundle of radius $R$ in $E_{m, k}$.  The mirror of $D_RE_{m, k}$ is
\[
D_RE_{m, k}^{\vee}:= \left\{(z_1, ..., z_{m+1})\in E^{\vee}_{m, k}\hspace{.1cm}\bigg|\hspace{.1cm} val(z_{m+1}) \leq \pi R^2\right\},
\] 
equipped with $\cg{W}\big|_{D_RE_{m, k}^{\vee}}$.  

For $I = (i_1, ..., i_{m+1})\in\RR^{m+1}$ and ${\bf z} = (z_1, ..., z_{m+1})$, denote $(z_1^{i_1}, ..., z_{m+1}^{i_{m+1}})$ by ${\bf z}^{I}$.  We denote the ring of functions on $D_RE_{m, k}^{\vee}$ in the variable ${\bf z}$ by $\cg{O}(D_RE_{m, k}^{\vee})_{{\bf z}}$, where
\[
\cg{O}(D_RE_{m, k}^{\vee})_{{\bf z}} = \left\{\sum_{i=0}^{\infty} c_i{\bf z}^{I_i}\hspace{.1cm}\bigg|\hspace{.1cm} c_i\in\Lambda;\hspace{.1cm} I_i\in\RR^{m+1};\hspace{.1cm} \lim_{i\rightarrow\infty} val(c_i{\bf z}^{I_i}) = \infty\hspace{.1cm}\forall\hspace{.1cm}{\bf z}\in D_RE_{m, k}^{\vee}\right\}.
\]

Let $A_{(a, b]}\subset \Lambda$ be the annulus $\{z\in\Lambda\hspace{.05cm}\big|\hspace{.05cm} val(z)\in(a, b]\}$.  A straight-forward computation shows that
\[
Jac(\cg{W}\big|_{D_RE_{m, k}^{\vee}}) \cong \bigslant{\cg{O}(A_{(0, \pi R^2]})_{z_{m+1}}}{(1 - (-k)^kTz_{m+1}^{-1 - m + k})}.
\]
If $\pi R^2 < \frac{1}{1 + m - k}$, then $val(Tz_{m+1}^{-1 - m + k}) > 0$ for all $z_{m+1}\in A_{(0, \pi R^2]}$.  It follows that $1 - (-k)^kTz_{m+1}^{-1-m+k}$ is a unit in $\cg{O}(A_{(0, \pi R^2]})_{z_{m+1}}$, and so 
\[
Jac(\cg{W}\big|_{D_RE_{m, k}^{\vee}}) = 0.  
\]
If $\pi R^2 \geq \frac{1}{1 + m - k}$ then 
\[
Jac(\cg{W}\big|_{D_RE_{m, k}^{\vee}})\cong\bigslant{\Lambda[z]}{(1 - (-k)^kTz^{-1-m+k}}).
\]  
Mirror symmetry now predicts
\begin{equation}
\widehat{SH^*}(D_RE_{m, k}; \Lambda)\cong\left\{
\begin{array}{cc}
 \bigslant{\Lambda[z]}{\left(1 - (-k)^kTz^{-1-m+k}\right)} & R \geq \frac{1}{\sqrt{\pi(1+m-k)}} \\
0 & R <  \frac{1}{\sqrt{\pi(1+m-k)}}
\end{array}\right.
\label{eq:conj}
\end{equation}
Note that Equation (\ref{eq:conj}) restricted to $k=1, m = 2$ matches the result of Theorem \ref{prop:example}.

Generalizing Theorem \ref{prop:example}, we have the following conjecture.
\conjannuli*
Conjecture \ref{conj:annuli} predicts that $\widehat{SH^*}(W; \Lambda)$ is non-zero if and only if $W$ contains the monotone, Floer-essential Lagrangian contained in the radius--$\frac{1}{\sqrt{\pi(1+m-k)}}$ sphere bundle; equivalently, if and only if the expected mirror $(W^{\vee}, \cg{W}\big|_{W^{\vee}})$ of $W$, defined by
\[
W^{\vee} := \left\{(z_1, ..., z_{m+1})\in E_{m, k}^{\vee}\hspace{.1cm}\big|\hspace{.1cm} \pi R_1^2 \leq val(z_{m+1})\leq \pi R_2^2\right\},
\]
contains the critical locus of $\cg{W}$.

\end{document}